\documentclass[11pt]{article}

\usepackage[a4paper,top=2.54cm,bottom=2.54cm,left=3.17cm,right=3.17cm,%
            includehead,includefoot]{geometry}

\usepackage{amsmath,amssymb,amsfonts,amsthm}
\usepackage{graphicx}
\usepackage{float}
\usepackage{subcaption}
\usepackage{xcolor}
\usepackage{booktabs}
\usepackage{array}
\usepackage{hyperref}
  \hypersetup{colorlinks,citecolor=blue,linkcolor=blue,breaklinks=true}
\usepackage{pgfplots}
\pgfplotsset{compat=1.18}
\usepackage{epsfig}
\usepackage{epstopdf}
\usepackage{makecell}
\usepackage{enumitem}
\usepackage[authoryear,round]{natbib}
\usepackage{yhmath} 
\usepackage{tikz}
\usepackage{fontawesome}
\usetikzlibrary{shapes,positioning,shadows,calc,fit,arrows.meta,decorations.pathreplacing}
\numberwithin{equation}{section}

\allowdisplaybreaks

\newtheorem{theorem}{Theorem}[section]
\newtheorem{lemma}{Lemma}[section]

\newtheorem{definition}{Definition}[section]

\newtheorem{remark}{Remark}
\newtheorem{assumption}{Assumption}[section]
\newcommand{\R}{\mathbb{R}}

\begin{document}

\title{Dynamic Pricing for a Two-Sided Data Market Platform}

\author{Lijun Bo\thanks{Email: lijunbo@xidian.edu.cn, School of Mathematics and Statistics, Xidian University, Xi'an, 710126, China}
\and
Dongfang Yang\thanks{Email: yangdf@stu.xidian.edu.cn, School of Mathematics and Statistics, Xidian University, Xi'an, 710126, China}
\and 
Yijie Huang\thanks{Email: yijie.huang@polyu.edu.hk, Department of Applied Mathematics, The Hong Kong Polytechnic University, Kowloon, Hong Kong, China}
}

\date{}

\maketitle

\begin{abstract}  
We study a continuous-time dynamic pricing problem for a data platform that purchases raw data from privacy-sensitive providers and sells data products to consumers. The platform controls both the acquisition price offered to providers and the selling price charged to consumers. Provider and consumer arrivals are modeled by point processes whose intensities depend on the platform’s current data stock, capturing feedback
between data accumulation and market participation. We formulate the platform’s problem as an infinite-horizon stochastic control problem with a jump-diffusion state process and derive the associated nonlinear integro-differential HJB equation. We prove that the value function is the unique viscosity solution, establish classical regularity under suitable conditions, and verify the optimal feedback pricing policy. Finally, we conduct numerical analyses to examine the influences of model parameters on the optimal pricing policies.\\



\end{abstract}

\section{Introduction}

The era of big data has witnessed a pronounced rise in the value of data. Data and data products now play essential roles across diverse domains from scientific research and commercial operations to policy formulation. As a result, data markets have emerged, where real-time information purchases can be made to inform decisions.
There are fundamental distinctions between traditional goods and data. First, data are virtual and replicable. Once acquired and processed, the same dataset can be packaged and resold to multiple downstream consumers at negligible marginal cost. Therefore, the commercialization of the data and related  products does not diminish the data volume held by the platform. Second, an individual data record often has limited standalone value; its value becomes significant only when combined with complementary data (cf. \cite{MW99}). Third, data are often time-sensitive. In many applications, especially those involving real-time analytics, prediction, and decision support, the value of data may depreciate as it becomes outdated. This perishability is particularly evident in the current era of data explosion (cf. \cite{JW18}). Demand for real-time data is growing at an accelerating pace. These characteristics make pricing and acquisition decisions in data markets substantially different from those in traditional product markets.

 Dynamic pricing has been widely studied in the context of traditional physical goods sales, as in \cite{GV94}, \cite{ZZ00}, \cite{BH24}. However, dynamic pricing in data markets remains relatively underexplored.
In the work of \cite{NZ20}, a contextual dynamic pricing mechanism is proposed for online data markets, where a query can be sold to different consumers at different times and the broker has the ability to adjust prices dynamically over time. \cite{XJ17} addresses the problem of revenue maximization for a data collector facing sequentially arriving data providers whose privacy valuations are unknown. Focusing on strategic data consumers who may repeatedly submit low bids in an attempt to drive down prices, \cite{C22} explores data pricing techniques designed to counteract such manipulative behaviors. Meanwhile, \cite{ZA21} investigates a pricing problem in the context of fresh data trading, where a destination user requests and pays for fresh data updates from a source provider, and data freshness is measured using the age of information (AoI) metric. 
Within the dynamic pricing setting, \cite{AD19} designed a data marketplace where multiple sellers supply data for sale and multiple buyers come with their own machine learning models dynamically. These studies provide important insights, but they typically focus on either the demand side or specific data-trading mechanisms, and are mostly formulated in discrete time. Less attention has been paid to the joint dynamic pricing problem faced by a data platform that must simultaneously procure raw data from privacy-sensitive providers and sell data products to consumers.

The rapid growth of data-driven economies has given rise to data platforms that intermediate between data providers and data consumers. These platforms facilitate the exchange of raw data from privacy-sensitive individuals or organizations (supply side) and refined data products to end users (demand side). Existing literature primarily examines the role of data intermediaries, whose necessity stems from their capacity to mitigate information asymmetry and incompleteness inherent in product markets. By centralizing transactions, platforms can reduce search frictions, enforce data quality standards, and align divergent privacy preferences. 
Within this market structure with monopolistic data intermediaries, issues like data ownership, 
acquisition policy and pricing strategies are thoroughly examined (\cite{BB22} and \cite{Y22}). As in  \cite{BM16}, the two-sided market trading platform determines both the purchasing price for data providers and the selling price for data consumers. Providers and consumers may accept or reject the bids based on their privacy valuation and willingness-to-pay. It allows data platform to buy raw data from providers, apply data analytics, and sell refined data, i.e. the standard, outlier-removal, and transformed data to consumers.  

A central theme in the literature is that data size is widely regarded as a primary driver of data value and model accuracy. Insufficient data volume inherently compromises the delivery of optimal data analytics service performance.
\cite{LY17} link data value to information entropy, which increases with data size and classification accuracy. \cite{NA16} propose a willingness-to-pay function that grows with data size. Additionally, \cite{D12} assert that model accuracy rises with dataset size, and numerous prior studies have assessed model quality based sorely on the size of the datasets utilized in their creation. \cite{SK22} incorporate the maximum achievable accuracy with commonly used utility function. It results a non-decreasing function with decreasing marginal accuracy and increases asymptotically towards the maximum achievable accuracy as the data size grows.

Motivated by these observations, this paper studies a continuous-time dynamic pricing problem for a monopolistic data platform with dynamically arriving data providers and consumers. The platform controls two prices: the acquisition price paid to data providers and the selling price charged to consumers. Data providers arrive randomly and decide whether to sell their data according to their privacy losses. Consumers also arrive randomly and decide whether to purchase data products according to their willingness to pay and the quality of the platform’s data product. The platform’s data stock evolves over time through depreciation, random fluctuations, and data contributions from providers.

To capture the endogenous interaction between data accumulation and market participation, we model provider and consumer arrivals using point processes with state-dependent intensities. In particular, the arrival rates depend on the platform’s current data stock. This specification reflects the fact that databases of different sizes may have different levels of attractiveness to both sides of the market. Unlike exogenous jump specifications commonly used in the literature, here the jump intensity is allowed to depend on the database size, making the jumps endogenous. This class of state-dependent processes is more realistic and adaptable in practice (see, e.g., \cite{DN16} and \cite{BH25}).  In our model, the platform’s database grows through the successful participation of providers, which not only influences the future arrival rates of both consumers and providers but also affects each consumer’s willingness to pay (via a data‑quality function). Thus, our framework integrates supply‑side dynamics, demand‑side dynamics, and the strategic pricing decisions of the platform in a unified stochastic control setting. 

Under the model setting, our objective is to find an optimal pair of pricing policies (the acquisition price for data providers and the selling price for consumers) that maximizes the platform's expected net profit over an infinite horizon. The state process is a jump‑diffusion with jump terms that depend on both the control and the current state. This problem reduces to solving the associated Hamilton–Jacobi–Bellman (HJB)  equation which is derived via the dynamic programming principle. However, the resulting HJB equation is of the integro-differential type and fully nonlinear due to the presence of a non‑local integral term arising from the jump component of the state process. Consequently, an analytical closed‑form solution is generally unattainable, and more sophisticated mathematical techniques are required. To overcome this difficulty, we follow the roadmap of \cite{DL13} and proceed in several steps to establish the existence of a classical solution to the integro‑differential HJB (ID‑HJB) equation. We first verify that the value function is the unique solution to the ID-HJB equation in the viscosity sense (see, e.g., \cite{BC08} and \cite{BI08}).  Next, we treat the non‑local integral term of the ID‑HJB equation as an inhomogeneous term that depends on the value function. By doing so, the original integro‑differential equation reduces to a non‑homogeneous second‑order ordinary differential equation (ODE). We then investigate the existence and uniqueness of the classical solution to this ODE, leveraging standard results from elliptic ODE theory. This step bridges the gap between viscosity solutions and classical smoothness. With a classical solution in hand, we present a verification theorem that formally confirms that this solution coincides with the true value function. Moreover, the verification theorem allows us to characterize the optimal pricing policies in feedback form: the optimal acquisition price and selling price are expressed as measurable functions of the current data volume. Finally, to complement the theoretical analysis and to gain practical insights into the behavior of the optimal pricing strategies, we provide numerical examples. Specifically, we specify functional forms for the intensity functions, privacy valuation distribution, willingness‑to‑pay distribution, cost function, and quality function, and then solve the resulting ODE using finite difference methods. The numerical results illustrate how model parameters affect the optimal acquisition and selling prices.  In particular, we uncover a data platform life-cycle strategy driven by an inverted-U value function. Our analysis reveals that platforms confront an optimal data stock beyond which further accumulation reduces profitability, triggering a strategic shift from aggressive procurement to cost control. Furthermore, we demonstrate the spillover effect between the consumer and provider sides, and quantify how processing costs asymmetrically impact pricing policies. These insights provide novel managerial guidelines for data platforms seeking to balance growth, data quality, and operational efficiency.

The remainder of this paper is organized as follows.  Section \ref{sec:formulation} formulates the optimal pricing problem in continuous time and derives the associated HJB equation. Section \ref{sec:viscosity} establishes with the well-posedness of the HJB equation in the viscosity solution sense. Section \ref{sec:verification} proves the existence
and uniqueness of classical solutions to the HJB equation, and characterizes the optimal pricing strategies through a verification theorem.  Section \ref{sec:example}, presents numerical examples and examines the effects of key model parameters on the optimal pricing policies.

\begin{figure}[h]\label{market}
\centering
\begin{tikzpicture}[
    node distance=1.5cm and 2cm,
    box/.style={
        rectangle, 
        draw=black, 
        fill=none, 
        text width=2.2cm,
        minimum height=1.5cm, 
        align=center,
        font=\sffamily\bfseries,
        rounded corners=5pt  
    },
    label_style/.style={
        font=\sffamily\Large, 
        text=blue!40!black
    },
    arrow/.style={
        -{Stealth[scale=1.2]}, 
        thick, 
        black, 
        font=\sffamily\bfseries
    },
    timeline_arrow/.style={
        -{Stealth[scale=1.5]}, 
        thick, 
        line width=1.2pt,
        gray!60
    },
     brace/.style={
        decorate,
        decoration={brace, amplitude=8pt, raise=2pt, mirror},
        thick,
        black
    }
]

     \node[box, minimum width=2.2cm, minimum height=5.5cm] (operator) at (0, -0.5) {\large\faServer\ \\Platform};

    \node[box, right=2.5cm of operator] (buyer_n) {\faUser\ \\Consumer 2};
    \node[box, above=0.8cm of buyer_n] (buyer_1) {\faCogs\ \\Consumer 1};
    \node[box, below=0.8cm of buyer_n] (buyer_N) {\faBarChart\ \\Consumer n};
    \node[box, left=2.5cm of operator] (seller_n) {\faMobile\ \\Provider 2};
    \node[box, above=0.8cm of seller_n] (seller_1) {\faBriefcase\ \\Provider 1};
    \node[box, below=0.8cm of seller_n] (seller_N) {\faTablet\ \\Provider n};

    \path (buyer_1) -- (buyer_n) node[midway, font=\Huge, text=blue!70!black] {$\vdots$};
    \path (buyer_n) -- (buyer_N) node[midway,  font=\Huge, text=blue!70!black]{$\vdots$};
    \node[below=0.1cm of buyer_N, font=\Huge, text=blue!70!black](buyer_inf) {$\vdots$};
    \path (seller_1) -- (seller_n) node[midway, font=\Huge, text=blue!70!black] {$\vdots$};
    \path (seller_n) -- (seller_N) node[midway, font=\Huge, text=blue!70!black] {$\vdots$};
    \node[below=0.1cm of seller_N, font=\Huge, text=blue!70!black](seller_inf) {$\vdots$};
    
    \draw[arrow]([yshift=-0.3cm]operator.east) -- ([yshift=-0.3cm]buyer_n.west)
        node[midway, below] {Product};
    
    \draw[arrow] ([yshift=0.3cm]operator.east) -- ([yshift=0.3cm]buyer_n.west) 
        node[midway, above] {Ask $q$};
    \draw[arrow] ([yshift=-0.3cm]seller_n.east) -- ([yshift=-0.3cm]operator.west) 
        node[midway, below] {Raw data};
    
    \draw[arrow] ([yshift=0.3cm]operator.west) -- ([yshift=0.3cm]seller_n.east) 
        node[midway, above] {Bid $p$};

     \draw[timeline_arrow] (-7, 2.5) -- (-7, -4) 
        node[below, font=\sffamily\bfseries, black] { Time  $t$};

     \draw[timeline_arrow] (7, 2.5) -- (7, -4) 
        node[below, font=\sffamily\bfseries, black] { Time  $t$};
        
    \draw[brace] (-6.2, -4.5) -- (-0.2, -4.5);
    \draw[brace] (0.2, -4.5) -- (6.2, -4.5);
    
   \node[font=\sffamily\bfseries, black] at (-3.5, -5.2) {Data acquisition};
   \node[font=\sffamily\bfseries, black] at (3.5, -5.2) {Data monetization};

\end{tikzpicture}
\caption{The structure of a data market featuring a monopolistic data platform.}
\end{figure}
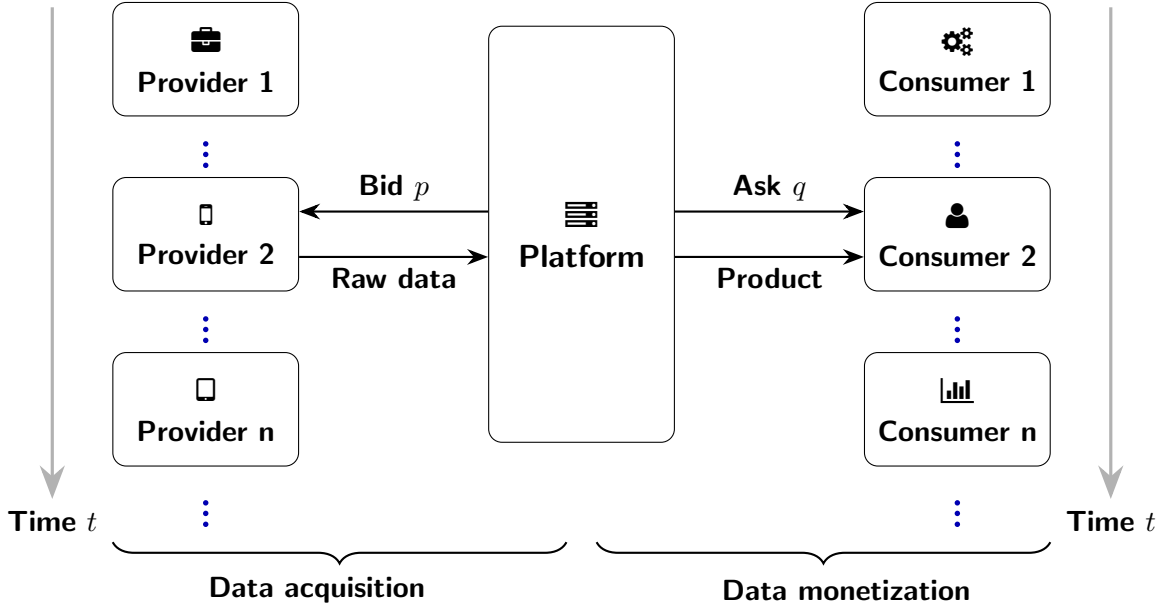

\section{Problem Formulation}\label{sec:formulation}

In this section, we formulate the dynamic optimal pricing problem for a data platform that operates as a monopolistic intermediary. The platform purchases raw data from privacy‑sensitive providers and sells refined data products to end users. We cast the problem as a continuous‑time stochastic control problem, in which the state process follows a jump‑diffusion with jump coefficients that depend on both the current state and the control.

Let $(\Omega, \mathcal{F}, \mathbb{F}, \mathbb{P})$ be a filtered probability space with the filtration $\mathbb{F}=(\mathcal{F}_t)_{t\geq 0}$ satisfying the usual conditions. Assume that this filtered probability space supports three random sources: a scalar Brownian motion $W=(W_t)_{t\geq0}$ and two Poisson point processes $N^{\rm dp}=(N^{\rm dp}_t)_{t\geq0}$ and $N^{\rm c}=(N^{\rm c}_t)_{t\geq0}$ with state-dependent intensity functions. ‌In the real-world data market, a monopolistic data platform operates by dynamically setting two prices, the price offered to data providers per unit of data supplied $p=(p_t)_{t\geq0}$, and the price charged to consumers per product executed $q=(q_t)_{t\geq0}$. Then, we introduce two main blocks (i) data acquisition and (ii) data monetization in our dynamic platform pricing framework:

$\bullet$ \textbf{Data acquisition:}
Let $X=(X_t)_{t\geq 0}$ be the total volume of data held on the platform with initial volume given by  $X_0 = x_0\in\R_+:=[0,\infty)$. The platform achieves data volume growth by purchasing data from providers.  
Data providers arrive according to a point process $N^{\rm dp}=(N_t^{\rm dp})_{t\geq 0}$, which is a doubly stochastic Poisson process with state-dependent intensity process $\lambda^{\rm dp}=(\lambda^{\rm dp}(X_t))_{t\geq0}$. Here, $\lambda^{\rm dp}(\cdot):\mathbb{R}\to\mathbb{R}_+$ is a measurable intensity function. Let $(T_k^{\rm dp})_{k=1}^{\infty}$ be the sequence of arrival times of the point process $N^{\rm dp}$.  Suppose that data sellers have different valuations for privacy. For $k\geq1$, upon arrival at time $T_k^{\rm dp}$, a provider draws a privacy (random) loss $\xi_k \sim \mu^{\rm dp}$  (\cite{AD19}), where $\mu^{\rm dp}\in{\cal P}(\mathbb{R}_+)$, i.e., it is a probability measure on $\mathbb{R}_+$. For $k\geq1$, the provider sells data if and only if the platform's offered price meets or exceeds this loss, i.e., $p_{T_k^{\rm dp}} \geq \xi_k$ (\cite{XJ15,XJ17}). Conditional on selling, the provider contributes a random data volume $Y_k \sim \nu\in{\cal P}(\mathbb{R}_+)$ for $k\geq1$.  Consequently, the data volume process $X=(X_t)_{t\geq0}$ accumulates according to the following dynamics:
\begin{align}\label{state}
X_t &=X_0-\int_0^t\delta X_sds+\sum_{k=1}^{N_t^{\rm dp}}{\bf1}_{\{p_{T_k^{\rm dp}}\geq \xi_k\}}Y_k+\sigma W_t,~~\forall t\geq0,
\end{align}
where, $\delta>0$ is the data depreciation rate and $\sigma W_t$ with volatility $\sigma>0$ represents the measurement errors at time $t$. Note that Eq. \eqref{state} can be rewritten as follows:
\begin{align}\label{state2}
X_t =X_0-\int_0^t\delta X_sds+\int_0^t\int_{\R_+\times \R_+} y{\bf1}_{\{z\leq \lambda^{\rm dp}(X_{s-})\mu^{\rm dp}(p_s)\}} \,\mathcal{N}(ds, dy,dz)+\sigma W_t,
\end{align}
where $\mathcal{N}(ds, dy, dz)$ is a Poisson random measure with compensator $ds\nu(dy)dz$. 

$\bullet$ \textbf{Data monetization:} 
Consumers arrive according to the Poisson point process $N^{\rm c}=(N_t^{\rm c})_{t\geq 0}$, which is also a doubly stochastic Poisson process with intensity process $\lambda^{\rm c}=(\lambda^{\rm c}(X_t))_{t\geq 0}$. Here, $\lambda^{\rm c}(\cdot):\mathbb{R}\to\mathbb{R}_+$ is a measurable intensity function. Let $(T_k^{\rm c})_{k=1}^{\infty}$ be the sequence of arrival times of the point process $N^{\rm c}$. For $k\geq 1$, the $k$-th consumer with single-unit demand for data product arrives at time $T_k^{\rm c}$, whose willingness-to-pay (WTP) for per unit quality is given by the random variable $\zeta_k\sim \mu^{\rm c}\in {\cal P}(\mathbb{R}_+)$. We can also refer to the random variable $\zeta_k$ as the value profile of consumer $k$ as in \cite{BB24}. Let $g(\cdot):\mathbb{R}\to (0, \infty)$ be a continuous function which measures the quality of the data product. Then, with $k\geq1$, the $k$-th consumer's WTP for a product of quality measured by $g(X_{T_k^{\rm c}})$ is given by $\zeta_kg(X_{T_k^{\rm c}})$ (cf. \cite{AD19} and \cite{NA16}). Thus, the $k$-th consumer purchases the data product if and only if the WTP is at least the posted price, i.e., $\zeta_kg(X_{T_k^{\rm c}}) \geq q_{T_k^{\rm c}}$. 

The platform's objective is to maximize its discounted net profit over the infinite horizon, which is given by, for any admissible pair of price policies $(p,q)=(p_t,q_t)_{t\geq0}\in{\cal U}$ with ${\cal U}$ being the admissible control set which will be specified later:
\begin{align}\label{objective}
J(x;p,q)&:=\mathbb{E}_x\Bigg[\sum_{k=1}^{\infty}e^{-\rho T_k^{\rm c}}q_{T_k^{\rm c}}{\bf1}_{\left\{q_{T_k^{\rm c}}\leq \zeta_kg(X_{T_k^{\rm c}})\right\}}-\sum_{k=1}^{\infty}e^{-\rho T_k^{\rm dp}}p_{T_k^{\rm dp}}{\bf1}_{\left\{p_{T_k^{\rm dp}}\geq \xi_k\right\}}Y_k\nonumber\\
&\quad\qquad-\int_0^{\infty}e^{-\rho t}\Phi(X_t)dt\Bigg],
\end{align}
where $\mathbb{E}_x[\cdot]:=\mathbb{E}[\cdot|X_0=x]$ for $x\in\mathbb{R}$, $\rho>0$ is the discount factor and  $\Phi(\cdot):\mathbb{R}\to\mathbb{R}$ is the cost function. The objective functional of the platform in \eqref{objective}
consists of three parts. The first two terms in the expectation respectively denote revenue generated from sales and payments made to data providers. The last term denotes the operational cost, encompassing expenses related to maintenance, storage and content moderation.
Equivalently, the objective functional \eqref{objective} can be rewritten as follows, for $(p,q)\in{\cal U}$,
\begin{align}\label{eq:Jxpq2nd}
J(x; p,q) = \mathbb{E}_x \left[ \int_0^\infty e^{-\rho t} F(X_t, p_t, q_t)dt \right],\quad \forall x\in\mathbb{R}.
\end{align}
Here, the running profit function $F(\cdot):\mathbb{R}\times \R_+^2\to\mathbb{R}$ is given by
\begin{align}\label{functioF}
F(x, p, q) := \lambda^{\rm c}(x) \left(1 - \mu^{\rm c}\left(\frac{q}{g(x)}\right)\right) q - \lambda^{\rm dp}(x) \mu^{\rm dp}(p) p \int_{0}^{\infty}y\nu(dy)-\Phi(x),
\end{align}
where $\mu^{\rm c}(x):=\mu^{\rm c}((0,x])$ and $\mu^{\rm dp}(x):=\mu^{\rm dp}((0,x])$ for $x\in\R_+$ are distribution functions of probability measures $\mu^{\rm c}\in{\cal P}(\R_+)$ and $\mu^{\rm dp}\in{\cal P}(\R_+)$, respectively.

We impose the following assumptions on model coefficients and parameters:
\begin{assumption}\label{ASS}
\begin{enumerate}\renewcommand{\labelenumi}{\rm(\roman{enumi})}
\item  The intensity function $\lambda^{\rm dp}(\cdot): \R\to\R_+$ is bounded and Lipchitz continuous. The quality function $g(\cdot):\R\to (0, \infty)$ is bounded and continuous.

\item The intensity function $\lambda^{\rm c}(\cdot):\R\to\R_+$ and the cost function $\Phi(\cdot):\R\to \R$  are continuous and satisfy that there exists $m\geq1$ and $L>0$ such that
$|\lambda^{\rm c}(x)|+|\Phi(x)| \leq L(1+|x|^m)$ for $x\in \R$. Moreover, $\Phi(\cdot):\R\to \R$ is convex. The probability distribution $\nu\in{\cal P}(\R_+)$ has the finite $2m$-order moment, i.e., $\int_0^{\infty}y^{2m}\nu(dy)<+\infty$.

\item The distribution function $x\to\mu^{\rm dp}(x)$ is continuous on $\R_+$, is $C^1$ and strictly increasing on $(0, b)$ with $b:=\inf\{x>0, \mu^{\rm dp}(x)=1\}$. Furthermore,  $\mu^{\rm dp}(x)$ has the finite right-derivative at $0$ and the finite left-derivative at $b$, which also satisfies that, for any closed interval $K\subset (0, b)$, there exists a constant $c_K>-1$ such that
\begin{align*}
\frac{\mu^{\rm dp}(x_2)}{(\mu^{\rm dp})'(x_2)}-\frac{\mu^{\rm dp}(x_1)}{(\mu^{\rm dp})'(x_1)}\geq c_K(x_2-x_1),\quad\forall x_1,x_2\in K,~ x_2>x_1.    
\end{align*}
The distribution function $x\to\mu^{\rm c}(x)$ is continuous and satisfies $\frac{1}{1-\mu^{\rm c}(x)}$ is strictly convex on $\{x\in\R_+;~\mu^{\rm c}(x)<1\}$ and $\lim_{x\to\infty }(1-\mu^{\rm c}(x))x=0$.
\item The discount rate $\rho$ satisfies $\rho>\rho_0:=\frac{\sigma^2}{2}\max\{m-1,1\}+\|\lambda^{\rm dp}\|_{\infty}(2^{\frac{m}{2}}\int_0^{\infty}(1+y^2)^{\frac{m}{2}}\nu(dy)-1)$, where $\|\lambda^{\rm dp}\|_{\infty}:=\sup_{x\in\R}|\lambda^{\rm dp}(x)|$.
\end{enumerate}

\end{assumption}

\begin{remark}
The convex structure of $x\to\Phi(x)$ in condition (ii) is economically grounded, reflecting the well-established principle that substantial investments in fixed infrastructure (e.g., storage space) entail progressively steeper marginal expenditures.

Condition (iii) on $\mu^{\rm c}\in{\cal P}(\R_+)$ and $\mu^{\rm dp}\in{\cal P}(\R_+)$ holds for several common distributions, such as the uniform and exponential distributions, which are frequently used in the data pricing literature (\cite{XJ17}, \cite{ZP17} and \cite{JW18}). A similar condition is adopted in \cite{CP21}. The limit condition $\lim_{x\to\infty }(1-\mu^{\rm c}(x))x=0$, also imposed in condition (iii), eliminates the possibility for platforms to make infinite profit by selling zero products at an infinite price (\cite{CP21}).

Under Assumption \ref{ASS}, one can easily verify that the running profit function $F(x,p,q)$ given by \eqref{functioF} is quasi-concave in $(p, q)\in \R_+^2$ for each fixed $x\in \R$.     Moreover, Assumption \ref{ASS} guarantees the existence and uniqueness of solutions to Eq.~\eqref{state2} via Picard iteration (see \cite{IW89}).
\end{remark}

 Let ${\cal U}$ be the set of all admissible pricing strategies which are $\mathbb{F}$-predictable processes $(p, q)=(p_t, q_t)_{t \geq 0}$ taking values on $\R_+^2$ with $\mathbb{E}[\int_0^t(p_s+q_s)ds]<\infty$ for any $t\geq 0$. The value function associated with the objective functional \eqref{eq:Jxpq2nd} is defined by, for any $x\in\R$,
\begin{align}\label{valuefunction2}
V(x):=\sup_{(p,q)\in {\cal U}}J(x; p, q)=\sup_{(p,q)\in{\cal U}}\mathbb{E}_x\left[\int_0^{\infty}e^{-\rho t}F(X_t, p_t,q_t)dt\right].
\end{align}
Then, we have that the value function $x\to V(x)$ satisfies a polynomial growth, which is provided in the following lemma:

\begin{lemma}\label{growth}
Let Assumption \ref{ASS} hold. Then, the value function  $x\to V(x)$ satisfies $|V(x)|\leq M(1+|x|^m)$ for any $x\in\R$. 
\end{lemma}

\begin{proof} For any admissible pair of pricing strategies $(p,q)\in {\cal U}$, let $X^x=(X_t^x)_{t\geq 0}$ be the controlled state process given by \eqref{state2} with $X_0^x=x\in \R$.  Applying It\^o's rule to $(1+|X_t^x|^2)^{\frac{m}{2}}$ for $m\geq 1 $, we have that
\begin{align}\label{eq:growth}
&\left(1+|X_t^x|^2\right)^{\frac{m}{2}}=(1+|x|^2)^{\frac{m}{2}}-\delta \int_0^t|X_s^x|^2m(1+|X_s|^2)^{\frac{m-2}{2}}ds\\
&\quad+\int_0^t\sigma m (1+|X_s^x|^2)^{\frac{m-2}{2}}dW_s+\frac{\sigma^2 m}{2}\int_0^t\left(1+|X_s^x|^2\right)^{\frac{m-4}{2}}\left((m-1)|X_s^x|^2+1\right)ds\nonumber\\
&\quad+\int_0^t\int_{\R_+^2}{\bf1}_{\{z\leq \lambda^{\rm dp}(X_{s-}^x)\mu^{\rm dp}(p_{s-}) \}}\left(\left(1+|X_{s-}^x+y|^2\right)^{\frac{m}{2}}-\left(1+|X_{s-}^x|^2\right)^{\frac{m}{2}}\right)\mathcal{N}(ds,dy,dz). \nonumber 
\end{align}
Note that, it holds that
\begin{align*}
\left(1+|X_s^x|^2\right)^{\frac{m-4}{2}}\left((m-1)|X_s^x|^2+1\right)&\leq \max\{m-1,1\}\left(1+|X_s^x|^2\right)^{\frac{m-2}{2}}\nonumber\\
&\leq \max\{m-1,1\}\left(1+|X_s^x|^2\right)^{\frac{m}{2}},
\end{align*}
and
\begin{align*}
\left(1+|X_{s-}^x+y|^2\right)^{\frac{m}{2}}-\left(1+|X_{s-}^x|^2\right)^{\frac{m}{2}}&\leq \left(1+2|X_{s-}^x|^2+2y^2\right)^{\frac{m}{2}}-\left(1+|X_{s-}^x|^2\right)^{\frac{m}{2}}\\
&\leq \left((1+y^2)^{\frac{m}{2}}2^{\frac{m}{2}}-1\right)\left(1+|X_{s-}^x|^2\right)^{\frac{m}{2}}.
\end{align*}
Taking expectations on both sides of \eqref{eq:growth} by a localized argument, and using the inequality $(1+|x|^2)^{\frac{m}{2}}\leq \max\{1, 2^{\frac{m-2}{2}}\}(1+|x|^m)$, we have
\begin{align*}
\mathbb{E}\left[\left(1+|X_t^x|^2\right)^{\frac{m}{2}}\right]&\leq (1+|x|^2)^{\frac{m}{2}}+\frac{\sigma^2}{2}\max\{m-1,1\}\int_0^t\mathbb{E}\left[\left(1+|X_s^x|^2\right)^{\frac{m}{2}}\right]ds\\
&\quad+\int_0^{\infty}\left((1+y^2)^{\frac{m}{2}}2^{\frac{m}{2}}-1\right)\nu(dy)\int_0^t\mathbb{E}\left[\lambda^{\rm dp}(X_{s-}^x)\left(1+|X_{s-}^x|^2\right)^{\frac{m}{2}}\right]ds\\
&\leq  \max\{1,	2^{\frac{m-2}{2}}\}(1+|x|^m)+\rho_0\int_0^t\mathbb{E}\left[\left(1+|X_{s}^x|^2\right)^{\frac{m}{2}}\right]ds
\end{align*}
with the constant $\rho_0:=\frac{\sigma^2}{2}\max\{m-1,1\}+\|\lambda^{\rm dp}\|_{\infty}(2^{\frac{m}{2}}\int_0^{\infty}(1+y^2)^{\frac{m}{2}}\nu(dy)-1)$. It can be deduced from the Gronwall's inequality that $\mathbb{E}[|X_t^x|^m]\leq \mathbb{E}[(1+|X_t^x|^2)^{\frac{m}{2}}]\leq 	C_m(1+|x|^m)e^{\rho_0t}$ for $C_m>0$ depending only on $m$. 

We first consider the null control $p_t = q_t = 0$ for $t \geq 0$. It follows from \eqref{eq:Jxpq2nd} and the moment estimate of $X$ that for $\rho >\rho_0$, 
\begin{align*}
	V(x) &\geq -\mathbb{E} \left[ \int_0^{\infty} e^{-\rho t} \Phi(X_t^x) dt \right] 
		   \geq -\int_0^{\infty} e^{-\rho t} L \left(1 + \mathbb{E}[|X_t^x|^m]\right) dt \\
		&\geq -\frac{L}{\rho} - L C_m \int_0^{\infty} e^{-(\rho-\rho_0)t} (1+|x|^m) dt  
		\geq -\frac{L}{\rho} - \frac{L C_m}{\rho-\rho_0}(1+|x|^m).
\end{align*}
For $q\in \R_+$, we define $h_1(q):=(1-\mu^{\rm c}(q))q$. From Assumption \ref{ASS}, we know $h_1(0)=0$ and $\lim_{q\to\infty}h_1(q)=0$. Choose $q_0\in\R_+$ such that $h_1(q_0)>0$. Since $\lim_{q\to\infty}h_1(q)=0$, there exists a sufficiently large constant $C>q_0$ such that $h_1(q)<h_1(q_0)$ for all $q>C$. Because $q\to h_1(q)$ is continuous in $\R_+$, there exists a point $\hat{q}\in [0,C]$ such that  $h_1(\hat{q})=\max_{q\in [0, C]}h_1(q)$. For any $q>C$, we have  $h_1(q)<h_1(q_0)\leq h_1(\hat{q})$. Hence, $\hat{q}$ is a global maximizer of $h_1$ on $\R_+$. Moreover, $h_1(\hat{q})=\sup_{q\in \R_+}h_1(q)<\infty$. Then there exists $C>0$, 
\begin{align*}
	F(x, p, q) &\leq \lambda^{\rm c}(x) \left(1 - \mu^{\rm c}\left(\frac{q}{g(x)}\right)\right) q - \Phi(x) \\
		&\leq \lambda^{\rm c}(x) g(x) \sup_{q\in \mathbb{R}_+} \left\{ \left(1 - \mu^{\rm c}(q)\right)q\right\} - \Phi(x) 
		\leq C(1+|x|^m).
\end{align*}
It follows from \eqref{eq:Jxpq2nd} that, for any $(p,q)\in{\cal U}$ and  $\rho>\rho_0$,
\begin{align*}
	J(x; p, q) \leq \int_0^{\infty} e^{-\rho t} C \left(1 + \mathbb{E}[|X_t^x|^m]\right) dt 
	\leq \frac{C}{\rho} + \frac{C C_m}{\rho-\rho_0}(1+|x|^m),
\end{align*}
where $C>0$ is a constant independent of $x, p, q$.
Due to the arbitrariness of $(p, q) \in {\cal U}$, we have $V(x)\leq \frac{C}{\rho} + \frac{C C_m}{\rho-\rho_0}(1+|x|^m)$. Thus, we can conclude that $|V(x)| \leq M(1+|x|^m)$ for some positive constant $M$.
\end{proof}

Using the dynamic programming principle (DPP), we have the following HJB equation satisfied by the value function $V(x)$ on $\R$ formally:
\begin{align}\label{eq:HJB}
\rho V(x) &= \frac{1}{2}\sigma^2V''(x)-\delta x V'(x)- \Phi(x)+ \lambda^{\rm c}(x)\sup_{q\in \R_+} \left\{ \left(1-\mu^{\rm c}\left(\frac{q}{g(x)}\right)\right)q \right\}  \nonumber \\
&\quad + \lambda^{\rm dp}(x)\sup_{p\in \R_+}\left\{\mu^{\rm dp}(p) \int_0^\infty \left( V(x+y) - V(x)-py \right) \nu(dy) \right\}.
\end{align}
The following lemma facilitates the characterization of the optimal prices, provided that the value function is a classical solution to the HJB equation, a result we will establish in the next section.
\begin{lemma}\label{control}
Let Assumptions \ref{ASS} hold. If the value function $x\to V(x)$ is a classical solution to the HJB equation \eqref{eq:HJB}, there exists a unique pair of measurable functions $x\to(p^*(x),q^*(x))$ such that, for $x\in\R$, 
\begin{align}\label{controlfuncition}
H_V^{({\rm dp})}(x; p^*)&=\sup_{p\in \R_+}H_V^{({\rm dp})}(x; p):=\sup_{p\in \R_+} \left\{\mu^{\rm dp}(p) \int_0^\infty \left(V(x+y) - V(x)-py \right) \nu(dy) \right\},\nonumber\\
H^{({\rm c})}(x; q^*)&=\sup_{q\in \R_+}H^{({\rm c})}(x; q):=\sup_{q\in \R_+}\ \left\{ \left(1-\mu^{\rm c}\left(\frac{q}{g(x)}\right)\right)q \right\}.
	\end{align}
   Moreover, the mapping $x\to p^*(x)$ is locally Lipschitz continuous, and both $H_V^{({\rm dp})}(\cdot; p^*(\cdot))$ and $H^{({\rm c})}(\cdot; q^*(\cdot))$ are continuous on $\R$.
\end{lemma}

\begin{proof}

Let $x\in\R$ be fixed. For a given $z \in \mathbb{R}$, define $h_2(p) := \mu^{\mathrm{dp}}(p)(z - p)$ for $p>0$. By using Assumption \ref{ASS}-(iii), $p\to h_2(p)$ is continuous on $\mathbb{R}_+$. Define the following auxiliary function by
\begin{align*}
f(p) := p + \frac{\mu^{\mathrm{dp}}(p)}{(\mu^{\mathrm{dp}})'(p)}, \quad \forall p \in (0, b).
\end{align*}
Assumption \ref{ASS}-(iii) implies that $f$ is strictly increasing on $(0, b)$ with $f(0^+) = 0$. Let $b_f := \lim_{p \to b^-} f(p) = \frac{1}{(\mu^{\rm dp})'(b^-)} + b\in (b, \infty]$. We now consider three cases.

\begin{itemize}
\item If $z\leq 0$, then $h_2(p)\leq 0$ for all $p\geq 0$, and the unique maximizer, denoted by $\hat{p}$, is $\hat{p}=0$.  

\item If $z \geq b_f$, $h_2'(p) = (\mu^{\rm dp})'(p)(z - f(p))$ for $p \in (0, b)$. Since $z \geq b_f > f(p)$ and $(\mu^{\rm dp})'(p) > 0$, we have $h_2'(p) > 0$. Thus $h_2(p)$ is strictly increasing on $[0, b]$. For $p \geq b$, since $\mu^{\rm dp}(p) = 1$, we have $h_2(p) = z - p$, which is strictly decreasing in $p$ with $h_2'(p) = -1 < 0$. Combining both, the function $h_2(p)$ increases on $[0, b]$ and decreases on $[b, \infty)$. Thus, the unique global maximizer on $\R_+$ is $\hat{p}(z) = b$.

\item If $z\in (0,b_f)$, by Assumption \ref{ASS}-(iii), there exists a unique $\hat{p} \in (0, b)$ such that $f(\hat{p}) = z$, which means $h_2'(\hat{p}) = 0$. For $p \in (0, \hat{p})$, $f(p) < f(\hat{p}) = z$ and $ h_2'(p) > 0$. For $p \in (\hat{p}, b)$, $f(p) > f(\hat{p}) = z $ and $h_2'(p) < 0$. For $p \geq b$, $h_2(p) = z - p$ is strictly decreasing in $p$, so $h_2'(p) = -1 < 0$. Therefore, $\hat{p}(z) = f^{-1}(z) \in (0, b)$ is the unique global maximizer of $p\to h_2(p)$ on $\R_+$, where $f^{-1}$ denotes the inverse function of $f$. 
\end{itemize}

Next, we show that the function $z\to \hat{p}(z)$ is locally Lipschitz continuous. Let $K \subset \R$ be an arbitrary compact interval and $z_1,z_2\in K$ with $z_1<z_2$. For $K\subset (0, b_f)$, the set $K':=f^{-1}(K)\subset (0, b)$ is also compact. By the local uniform monotonicity of $f$ given by Assumption \ref{ASS}-(iii), there exists $c_K>-1$ such that
\begin{align*}
z_2-z_1 = f(\hat{p}(z_2)))-f(\hat{p}(z_1))) \geq (c_K+1)(\hat{p}(z_2))-\hat{p}(z_1))).    
\end{align*}
Thus, it holds that $|\hat{p}(z_2)-\hat{p}(z_1)|=|f^{-1}(z_2)-f^{-1}(z_1)| \leq \frac{1}{c_K+1}|z_2-z_1|$.
For $K\subset (-\infty, 0)$ and $K\subset (b_f, \infty)$, $\hat{p}(z_1)=\hat{p}(z_2)$. For $z_1 \le 0$ and $z_2 \in (0, b_f)$, we have $\hat{p}(z_1)=0$ and $f(\hat{p}(z_2)) = z_2$, which yields that
\begin{align*}
|\hat{p}(z_2)-\hat{p}(z_1)|= |\hat{p}(z_2)-0|\leq z_2\leq |z_2-z_1|.    
\end{align*}
For $z_1 \in (0, b_f)$ and $z_2 \geq b_f$, it holds that $\hat{p}(z_2) = b$ and $f(\hat{p}(z_1)) = z_1$. By the local uniform monotonicity of $f$, there exists $c_K > -1$ such that 
\[|z_2 - z_1| = z_2 - z_1 \geq b_f - z_1 =f(b^-) - f(\hat{p}(z_1)) \geq (c_K+1)(b - \hat{p}(z_1)) = (c_K+1)|\hat{p}(z_2) - \hat{p}(z_1)|,\]
which implies $|\hat{p}(z_2) - \hat{p}(z_1)| \le \frac{1}{c_K+1}|z_2 - z_1|$.
Therefore, $z\to\hat{p}(z)$ is locally Lipschitz continuous on $\R$. 

Let us introduce $z(x):=\frac{\int_0^\infty \bigl(V(x+y)-V(x)\bigr)\nu(dy)}{\int_0^\infty y\,\nu(dy)}$.
Since $V\in C^2(\mathbb R)$ and $\int_0^\infty y\,\nu(dy)<\infty$, the mapping $x\to z(x)$ is locally Lipschitz continuous. Consequently, $x\to p^*(x)	=\hat{p}(z(x))$ is locally Lipschitz continuous  (hence measurable) as a composition of two locally Lipschitz mappings. Moreover, $x\to H_V^{\rm (dp)}(x;p^*)$ is continuous on $\R$.
    
On the other hand, we have obtained the existence of maximizers of $q\to h_1(q)$ on $\R_+$ in the proof of Lemma \ref{growth}. By Example 3.38 in \cite{BV04}, the function $ q\to h_1(q)$ is strictly quasi-concave on $\{q\in \R_+;~ \mu^{\rm c}(q) < 1\}$. Note that $\lim_{q\to\infty}h_1(q)=0$ and $h_1(0)=0$ (see Fig.\ref{fig:combinedh1}), the existence and uniqueness of the maximizer, denoted by $\hat{q}$, is established.  The same conclusion holds for $q\to H^{({\rm c})}(x;q)$ since $H^{({\rm c})}(x;q)=h_1\left(\frac{q}{g(x)}\right)g(x)$ for $q\in \R_+$ and $x\to g(x)$ is independent of the variable $q$. Denote by $q^*(x):=\hat{q}g(x)$ the unique maximizer of $q\to H^{\rm (c)}(x;q)$. Then, $x\to H^{\rm (c)}(x;q^*(x))$ is continuous on $\R$.
Thus, the proof of the lemma is complete.
\end{proof}

\begin{figure}[ht]
\centering

\begin{tikzpicture}
\begin{axis}[
    width=9cm,
    height=6cm,
    xmin=0,
    xmax=10,
    ymin=-0.1,
    ymax=0.55,
    axis lines=middle,
    xtick=\empty,
    ytick=\empty,
    xlabel={$q$},
    ylabel={$h_1(q)$},
    xlabel style={anchor=west},
    ylabel style={anchor=south},
    axis line style={-stealth},
    samples=500,
    smooth,
    thick,
]

\addplot[blue, domain=0:10] {x*exp(-0.8*x)};
\addplot[only marks, mark=*, black] coordinates {(1.25,0.46)};
\addplot[dashed, red] coordinates {(1.25,0.46) (1.25,0)};
\addplot[only marks, mark=*] coordinates {(1.25,0)};
\node at (axis cs:1.25,-0.02) [anchor=north] {$q^*$};
\addplot[dashed, gray] coordinates {(0,0) (10,0)};
\addplot[only marks, mark=*] coordinates {(6,0)};
\end{axis}
\end{tikzpicture}

\caption{The curve of $q\to h_1(q):=(1-\mu^{\rm c}(q))q$ on $\R_+$.}
\label{fig:combinedh2}
\end{figure}
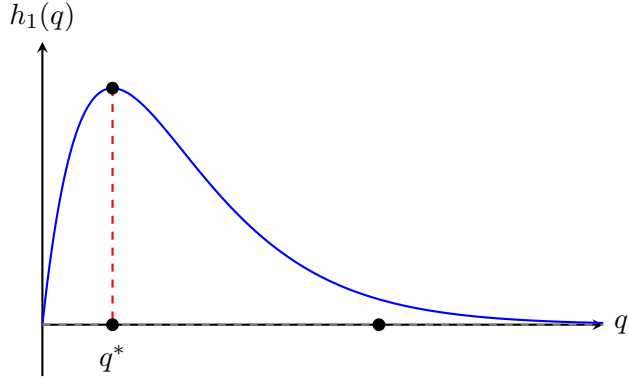

\begin{figure}[ht]
\centering

\begin{tikzpicture}
\begin{axis}[
    width=9cm,
    height=6cm,
    xmin=0,
    xmax=4,
    ymin=-0.8,
    ymax=1.6,
    axis lines=middle,
    xtick=\empty,
    ytick=\empty,
    xlabel={$p$},
    ylabel={$h_2(p)$},
    xlabel style={anchor=west},
    ylabel style={anchor=south},
    axis line style={-stealth},
    samples=500,
    smooth,
    thick,
]

\addplot[blue, domain=0:4] {(1-exp(-4*x))*(2.2-x)};
\addplot[dashed, gray] coordinates {(0,0) (4,0)};
\addplot[only marks, mark=*] coordinates {(2.2,0)};
\node at (axis cs:2.2,-0.05) [anchor=north] {$z$};
\addplot[only marks, mark=*] coordinates {(3,0)};
\addplot[only marks, mark=*, black] coordinates {(0.5,1.47)};
\addplot[dashed, red] coordinates {(0.5,1.47) (0.5,0)};
\addplot[only marks, mark=*] coordinates {(0.5,0)};
\node at (axis cs:0.55,-0.05) [anchor=north] {$p^*$};
\end{axis}
\end{tikzpicture}
\caption{The curve of $p\to h_2(p):=\mu^{\rm dp}(p)(z-p)$ on $\R_+$.}
\label{fig:combinedh1}
\end{figure}
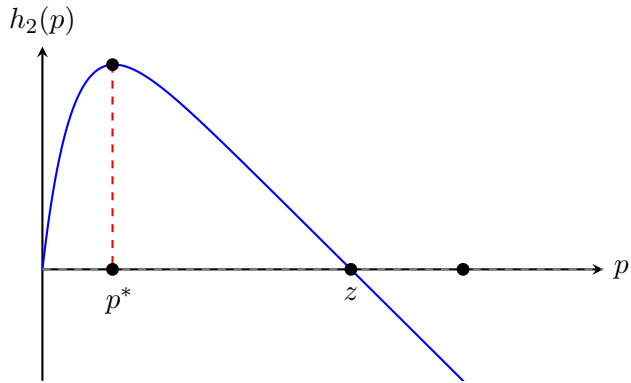

\section{Viscosity Solution}\label{sec:viscosity}

In this section, we study the well-posedness of the  integro-differential HJB (ID-HJB) equation \eqref{eq:HJB} within the framework of viscosity solutions. The need for this framework arises from the presence of the non‑local integral operator in \eqref{eq:HJB}, which prevents us from directly applying classical PDE theory. Viscosity solutions allow us to handle fully nonlinear and non‑local equations without requiring the value function to be a priori differentiable.

To proceed, let us introduce the following differential operator ${\cal L}$ acted on $\phi\in C^2(\R)$ and non-local integral operator ${\cal I}$ acted on $\psi\in C(\R)$, respectively, for $x\in\R$, 
\begin{align}
\mathcal{L}\phi(x)&:=\frac{1}{2}\sigma^2\phi''(x)-\delta x\phi'(x)- \Phi(x)+ \lambda^{\rm c}(x)\sup_{q\in \R_+} \left\{ \left(1-\mu^{\rm c}\left(\frac{q}{g(x)}\right)\right)q \right\},\label{eq:operatorL}\\
\mathcal{I}\psi(x)&:=\lambda^{\rm dp}(x)\sup_{p\in \R_+}\left\{\mu^{\rm dp}(p) \int_0^\infty \left( \psi(x+y) - \psi(x)-py \right) \nu(dy) \right\}.\label{eq:operatorI}
\end{align}
We now give the definition of viscosity solutions to ID-HJB equation \eqref{eq:HJB}:
\begin{definition}\label{def1}
\begin{enumerate}
\item[{\rm(i)}] A lower semi‑continuous (l.s.c) function $x\to v(x)$ is a viscosity supersolution of Eq.~\eqref{eq:HJB} if, for any $\bar{x}\in\R$ and test function $\varphi\in C^2(\R)$ such that $v-\varphi$ attains a minimum at $\bar{x}$ with $(v-\varphi)(\bar{x})=\min_{x\in \R}(v-\varphi)(x)$, it holds that $\rho v(\bar{x})-\mathcal{L}\varphi(\bar{x})-\mathcal{I}\varphi(\bar{x})\geq 0$. 

\item[{\rm(ii)}] An upper semi‑continuous (u.s.c) function $x\to v(x)$ is a viscosity subsolution of Eq.~\eqref{eq:HJB} if, for any $\bar{x}\in\R$ and test function $\varphi\in C^2(\R)$ such that $u-\varphi$ attains a maximum at $\bar{x}$ with $(v-\varphi)(\bar{x})=\max_{x\in \R}(v-\varphi)(x)$, it holds that $\rho v(\bar{x})-\mathcal{L}\varphi(\bar{x})-\mathcal{I}\varphi(\bar{x})\leq 0$ . 

\item[{\rm(iii)}] A locally bounded function $x\to v(x)$ is a viscosity solution of Eq.~\eqref{eq:HJB} if its u.s.c. envelope $x\to v^*(x)$ and its l.s.c. envelope $x\to v_*(x)$ are respectively the subsolution and supersolution to Eq.~\eqref{eq:HJB}.
\end{enumerate}
\end{definition}

Then, we have the following result on the existence of viscosity solutions to ID-HJB equation \eqref{eq:HJB}.
\begin{lemma}\label{existence}
Let Assumption \ref{ASS} hold. Then, the value function $x\to V(x)$ defined by \eqref{valuefunction2} is a viscosity solution to the ID-HJB equation \eqref{eq:HJB}. 
\end{lemma}

\begin{proof}
The proof is divided into two steps. 

{\it Step 1: viscosity supersolution.} By using Lemma \ref{growth}, the value function $x\to V(x)$ satisfies the polynomial growth condition, and hence $x\to V(x)$ is locally bounded on $\R$. Thus, we can define its l.s.c. envelope given by $x\to V_*(x):=\liminf_{\varepsilon\to 0}\{V(y): |y-x|\leq \varepsilon\}$. Let $\bar{x}\in\R$ and $\varphi\in  C^2(\R)$ such that $0=(V_*-\varphi)(\bar{x})=\min_{x\in \R }(V_*-\varphi)(x)$. Without loss of generality,  we can assume that the above inferior is strict, i.e.,$(V_*-\varphi)(x)>0$ for $x\neq \bar{x}$, otherwise, we can replace $\varphi$ by  $\varphi(x)+\varepsilon|x-\bar{x}|^2$ with sufficiently small $\varepsilon>0$. By definition of $V_*(\bar{x})$, there exists a sequence $(x_n)_{n\geq1}\subset\R$ such that $x_n\to \bar{x}$ and $V(x_n)\to V_*(\bar{x})$ as $n\to\infty$. From the continuity of $x\to\varphi(x)$ (since $\varphi\in C^2(\R)$), we can deduce that $\gamma_n:=V(x_n)-\varphi(x_n)\to 0$ as $n\to\infty$. Choose a pair of constant control strategy $(p,q)\equiv({\rm p},{\rm q})$ with $({\rm p},{\rm q})\in \R_+^2$. Denote by $X^{(n)}=(X_t^{(n)})_{t\geq 0}$ the corresponding (controlled) state process satisfying \eqref{state2} under this constant control strategy with $X_0^{(n)}=x_n$. Let $(h_n)_{n\geq 1}$ be a positive sequence such that $h_n\to 0$ and $\frac{\gamma_n}{h_n}\to 0$ as $n\to \infty$. Introduce a sequence of $\mathbb{F}$-stopping time $\tau_n:=\inf\{t\geq 0;~|X_t^{(n)}-x_n|\geq \eta\}\wedge h_n$, where $\eta>0$ is a fixed constant. Applying the DPP from $0$ to $\tau_n$, we have
\begin{align*}
V(x_n)\geq \mathbb{E}\left[ \int_0^{\tau_n}e^{-\rho t}F(X_t^{(n)}, p, q)dt+e^{-\rho \tau_n}V(X_{\tau_n-}^{(n)})\right].    
\end{align*}
Note that $V\geq V_*\geq \varphi$ on $\R$. Then, we obtain
\begin{align*}
\varphi(x_n)+\gamma_n\geq \mathbb{E}\left[ \int_0^{\tau_n}e^{-\rho t}F(X_t^{(n)}, p, q)dt+e^{-\rho \tau_n}\varphi(X_{\tau_n-}^{(n)})\right].    
\end{align*}
Applying It\^o's formula to $e^{-\rho t}\varphi(X_t^{(n)})$,  we have 
\begin{align}\label{eq:step1}
0\leq\frac{\gamma_n}{h_n}+\frac{1}{h_n}\mathbb{E}\left[\int_0^{\tau_n}e^{-\rho\tau_n}f_{\varphi}(X_t^{(n)}, p, q)dt\right],    
\end{align}
where, the function $f_{\varphi}:\R\times \R_+^2\to\R$ is defined by 
\begin{align*}
f_{\varphi}(x, p,q)&:=	\rho \varphi(x)-\frac{1}{2}\sigma^2\varphi''(x)
		+\delta x \varphi'(x)+ \Phi(x)-\lambda^{\rm c} (x) \left(1-\mu^{\rm c}\left(\frac{q}{g(x)}\right)\right) q  \\
		&\quad- \lambda^{\rm dp}(x)\mu^{\rm dp}(p) \int_0^\infty ( \varphi(x+y) - \varphi(x)-py ) \nu(dy).
\end{align*}
Note that, a.s.
\begin{align*}
\sup_{t\in [0,h_n]}\left|X_t^{(n)}-x_n\right|^2&\leq C\Bigg\{(1+x_n^2)h_n+\int_0^{h_n}|X_s^{(n)}-x_n|^2ds+\sup_{t\in [0,h_n]}|\sigma W_t|^2\\&\quad+\sup_{t\in [0,h_n]}\left|\int_0^t\int_{\R_+\times\R_+}y{\bf1}_{\{z\leq \lambda^{\rm dp}(X_{s-}^{(n)}\mu^{\rm dp}(p))\}}\widetilde{\mathcal{N}}(ds,dy,dz)\right|^2\Bigg\}.
\end{align*}
where $\tilde{\cal N}(ds,dy,dz):={\cal N}(ds,dy,dz)-\nu(dy)dzds$ is the compensated Poisson random measure. It follows from the Burkholder-Davis-Gundy (BDG) inequality that
\begin{align*}
\mathbb{E}\left[\sup_{t\in [0,h_n]}|X_t^{(n)}-x_n|^2\right]&\leq C\left\{(1+x_n^2)h_n+\int_0^{h_n}\mathbb{E}\left[\sup_{s\in [0,t]}\left|X_s^{(n)}-x_n\right|^2\right]dt\right\}.
\end{align*}
By Chebyshev's inequality and Gronwall's lemma, we have 
\begin{align}\label{cheb}
\mathbb{P}\left(\sup_{t\in [0, h_n]}\left|X_t^{(n)}-x_n\right|\geq \eta\right)&\leq \frac{\mathbb{E}\left[\sup_{t\in [0, h_n]}\left|X_t^{(n)}-x_n\right|^2\right]}{\eta^2}\nonumber\\
&\leq \frac{1}{\eta^2}C(1+x_n^2)e^{C h_n}h_n\to 0,~~ {\rm as}~n\to\infty.
\end{align}
As in \cite{BH25}, we define the event
$\Omega_{n,\epsilon}:=\{ \sup_{ t \in [0, h_n]} |X_t^{(n)} - x_n| \leq \epsilon\}$ for $\epsilon > 0$, then $\mathbb{P}(\Omega_{n,\epsilon}^c)\to0$ as $n \to \infty$. Note that, we have
\begin{align*}
\frac{1}{h_n} \mathbb{E} \left[ \int_0^{\tau_n} e^{-\rho \tau_n} f_{\varphi}(X_t^{(n)}, p, q) \, dt \right] 
&= \frac{1}{h_n} \mathbb{E} \left[ \int_0^{h_n} e^{-\rho h_n} f_{\varphi}(X_t^{(n)}, p, q) \, dt \cdot \mathbf{1}_{\Omega_{n,\epsilon}} \right]\\
&\quad + \frac{1}{h_n} \mathbb{E} \left[ \int_0^{\tau_n} e^{-\rho \tau_n} f_{\varphi}(X_t^{(n)}, p, q) \, dt \cdot \mathbf{1}_{\Omega_{n,\epsilon}^c} \right]. 
\end{align*}
On $ \Omega_{n,\epsilon} $, we have $ \tau_n = h_n $ and  $ f_{\varphi} $ is bounded on this compact set. Thus,  on $\Omega_{n,\epsilon}$, by using the mean value theorem, one has $\frac{1}{h_n} \int_0^{h_n} e^{-\rho h_n} f_{\varphi}(X_t^{(n)}, p, q) dt\to f_{\varphi}(\bar{x}, p, q)$ for $(p,q)\in \R_+^2$  a.s. as $n\to\infty$. On $\Omega_{n,\epsilon}^c $, we have from the definition of $\tau_n$ that $f_{\varphi}$ is also bounded.
As $\epsilon$ converges to 0 and $n$ converges to infinity on both sides of \eqref{eq:step1}, it holds that $f_{\varphi}(\bar{x}, p, q)\geq 0$ for all $(p,q)\in \R_+^2$. Hence, we have $\rho V(\bar{x})-(\mathcal{L}+\mathcal{I})\varphi(\bar{x})\geq 0$ on $\R$. 

{\it Step 2: viscosity subsolution.} To do it, 
we define $x\mapsto V^*(x):=\limsup_{\varepsilon\to 0}\{V(y): |y-x|\leq \varepsilon\}$ as the u.s.c. envelope of value function $x\mapsto V(x)$. Let $\bar{x}\in \R $, and $\varphi\in  C^2(\R)$ such that $0=(V^*-\varphi)(\bar{x})=\max_{x\in \R }(V^*-\varphi)(x)$. Similarly,  we can suppose that the previous maximum is strict. We will assume that $\rho V(\bar{x})-(\mathcal{L}+\mathcal{I})\varphi(\bar{x})> 0$ and verify through contradiction. By the continuity of $\varphi$, there exist $\eta>0$ and $\varepsilon>0$ such that $\rho \varphi (x)-(\mathcal{L}+\mathcal{I})\varphi(x)\geq \varepsilon$ for all $x\in B(\bar{x},\eta)$. There exists a sequence $(x_n)_{n\geq 1}$ such that $x_n\to \bar{x}$ and $V(x_n)\to V^*(\bar{x})$ as $n\to\infty$. By the continuity of $\varphi$,  we have $\gamma_n:=V(x_n)-\varphi(x_n)\to 0$ as $n\to\infty$. For a positive sequence $(h_n)_{n\geq 1}$ satisfying $h_n\to 0$ and $\frac{\gamma_n}{h_n}\to 0$ as $n\to \infty$, there exists a pair of $\frac{\varepsilon h_n}{2}$-optimal control strategies $(\hat{p}, \hat{q})\in {\cal U}$ such that 
\begin{align*}
V(x_n)-\frac{\varepsilon h_n}{2}\leq \mathbb{E}\left[\int_0^{\theta_n}e^{-\rho t}f(X_t^{(n)}, \hat{p}_t,\hat{q}_t)dt+e^{-\rho \theta_n}V(X_{\theta_n-}^{(n)})\right].    
\end{align*}
Then, by $V\leq \varphi$, it holds that
\begin{align*}
 \varphi(x_n)+\gamma_n-\frac{\varepsilon h_n}{2}\leq \mathbb{E}\left[\int_0^{\theta_n}e^{-\rho t}f(X_t^{(n)}, \hat{p}_t,\hat{q}_t)dt+e^{-\rho \theta_n}\varphi(X_{\theta_n-}^{(n)} )\right]   
\end{align*}
in which $\theta_n:=\tau_n'\wedge h_n$ and $\tau_n':=\inf\{t\geq 0;~|X_t^{(n)}-x_n|\geq \eta'\}$. By applying It\^o's rule to $e^{-\rho t}\varphi(X_t^{(n)})$ again,  we have 
\begin{align}\label{eq:visco}
\frac{\gamma_n}{h_n}-\frac{\varepsilon}{2}+\frac{1}{h_n}\mathbb{E}\left[\int_0^{\theta_n}e^{-\rho\theta_n}f(X_t^{(n)}, \hat{p}_t, \hat{q}_t)dt\right]\leq 0.
\end{align}
Furthermore, note that, a.s., 
\begin{align*}
f(X_t^{(n)}, \hat{p}_t, \hat{q}_t)\geq \rho \varphi(X_t^{(n)})-(\mathcal{L}+\mathcal{I})\varphi(X_t^{(n)})\geq\varepsilon,\quad \text{on}~t\in[0,\theta_n]. 
\end{align*}
We deduce from \eqref{eq:visco} that $\frac{\gamma_n}{h_n}-\varepsilon\left(\frac{1}{2}-\frac{1}{h_n}\mathbb{E}[\theta_n]\right)\leq 0$ for all $n\geq1$. Then, it follows from \eqref{cheb} that
\begin{align*}
 \mathbb{P}(\tau_n'\leq h_n)\leq \mathbb{P}\left(\sup_{t\in [0, h_n]}|X_t^{(n)}-x_n|\geq \eta\right)\to 0,\quad n\to\infty.   
\end{align*}
Moreover, since $\mathbb{P}(\tau_n'>h_n)\leq \frac{1}{h_n}\mathbb{E}[\theta_n]\leq 1$ for all $n\geq1$, this implies that $\frac{1}{h_n}\mathbb{E}[\theta_n]$ tends to 1 as $n\to \infty$. We thus get the desired contradiction.
\end{proof}

To establish the uniqueness result, we next prove the comparison result for viscosity solutions to \eqref{eq:HJB}.
\begin{lemma}\label{comparison1}
Let $u(x)$ and $v(x)$ for $x\in\R$ be a viscosity subsolution and a viscosity supersolution to the HJB equation \eqref{eq:HJB}, both satisfying the polynomial growth condition. Then, we have $u\leq v$ on $\R$.
\end{lemma}

\begin{proof}
We show it by contradiction, and assume that there exists $\bar{x}\in \R$ such that $u(\bar{x})-v(\bar{x})\geq 2\kappa$ for a positive constant $\kappa$. Define $\Psi_{n, \epsilon}(x,\tilde{x}):=u(x)-v(\tilde{x})-\psi_{n, \epsilon}(x,\tilde{x})$ for any $(x, \tilde{x})\in\R^2$. Here,  the function $\psi_{n, \epsilon}(x,\tilde{x}):=n|x-\tilde{x}|^2+\epsilon(|x|^{2m}+|\tilde{x}|^{2m})$ with $(n, \epsilon)\in \R_+\times (0,1]$ and $m$ is given in Assumption \ref{ASS}. Furthermore, define the constant $M_{n, \epsilon}:=\sup_{(x,\tilde{x})\in \R^2}\Psi_{n, \epsilon}(x,\tilde{x})$. Note that $x\to u(x)$ and $x\to v(x)$ satisfy the polynomial growth condition. Then, by the upper semi-continuity of $(x,\tilde{x})\to\Psi_{n, \epsilon}(x,\tilde{x})$, we have $M_{n, \epsilon}<\infty$, and there exists $(x_{n,\epsilon},\tilde{x}_{n,\epsilon})\in \R^2$ such that $M_{n, \epsilon}=\Psi_{n,\epsilon}(x_{n,\epsilon},\tilde{x}_{n,\epsilon})$. This yields that $ M_{n,\epsilon}\geq u(\bar{x})-v(\bar{x})-\psi_{n,\epsilon}(\bar{x}, \bar{x})\geq 2\kappa -2\epsilon |\bar{x}|^{2m}$. This implies that there exists $\epsilon_0\in (0, \frac{\kappa}{2|\bar{x}|^{2m}})$ such that $M_{n,\epsilon}>\kappa$ for all $\epsilon \in (0, \epsilon_0]$. Using the fact $\Psi_{n,\epsilon}(0,0)\leq \Psi_{n,\epsilon}(x_{n,\epsilon},\tilde{x}_{n,\epsilon})$, together with the polynomial growth of $x\to u(x)$ and $x\to v(x)$, there exists a constant $C>0$ independent of $n,\epsilon$ such that
\begin{align*}
n|x_{n, \epsilon}-\tilde{x}_{n,\epsilon}|^2+\epsilon(|x_{n,\epsilon}|^{2m}+|\tilde{x}_{n,\epsilon}|^{2m})&\leq u(0)-v(0)+u(x_{n,\epsilon})-v(\tilde{x}_{n,\epsilon})\nonumber\\
&\leq u(0)-v(0)+2C(1+|x_{n,\epsilon}|^m+|\tilde{x}_{n,\epsilon}|^m).
\end{align*}
As the term in the left hand side will grow faster than the right when $x_{n ,\epsilon}$ and $\tilde{x}_{n ,\epsilon}$ are sufficiently large, we can find a positive constant $C_{\epsilon}>0$ indepedent of $n$ such that $|x_{n,\epsilon}|\vee |\tilde{x}_{n,\epsilon}|\leq C_{\epsilon}$ for all $n\geq 1$. By this point, there exists a subsequence, still denoted by $(x_{n,\epsilon},\tilde{x}_{n,\epsilon})_{n\geq 1}$, which converges to $(x_{\epsilon}, \tilde{x}_{\epsilon})\in \R^2$ as $n\to\infty$. Hence, $n|x_{n,\epsilon}-\tilde{x}_{n,\epsilon}|^2\leq 2C(1+2|C_{\epsilon}|^{2m})+u(0)-v(0)$. Consequently, we can conclude that $x_{n,\epsilon}-\tilde{x}_{n,\epsilon}\to 0$ as $n\to \infty$, and hence $x_{\epsilon}=\tilde{x}_{\epsilon}$. On the other hand, it follows from the fact $\Psi(x_{\epsilon},\tilde{x}_{\epsilon})\leq \Psi(x_{n,\epsilon},\tilde{x}_{n,\epsilon})$ for all $n\geq 1$ and $\epsilon>0$, we obtain
\begin{align*}
n|x_{n,\epsilon}-\tilde{x}_{n,\epsilon}|^2&\leq u(x_{n,\epsilon})-u(x_{\epsilon})+v(\tilde{x}_{n,\epsilon})-v(\tilde{x}_{\epsilon})+\epsilon (|x_{\epsilon}|^{2m}+|\tilde{x}_{\epsilon}|^{2m})\nonumber\\
&\quad-\epsilon(|x_{n,\epsilon}|^{2m}+|\tilde{x}_{n,\epsilon}|^{2m}).  
\end{align*}
By the semi-continuity of $x\to u(x)$ and $x\to v(x)$, one has $n|x_{n,\epsilon}-\tilde{x}_{n,\epsilon}|^2\to 0$ as $n\to\infty$.
As a result, by the construction of definition of $(x_{n,\epsilon},\tilde{x}_{n,\epsilon})$, we have $x_{n,\epsilon}$ is a local maximum of $x\to u(x)-\psi_{n, \epsilon}(x, \tilde{x}_{n,\epsilon})$ and $\tilde{x}_{n,\epsilon}$ is a local maximum of $\tilde{x}\to v(\tilde{x})+\psi_{n , \epsilon}(x_{n,\epsilon}, \tilde{x})$. By the Crandall–Ishii lemma (cf. \cite{CI92} and \cite{P09}), we have, for any $\tilde{\eta}>0$, there exist constants $a_{n,\epsilon}, b_{n,\epsilon}\in\R$ such that
\begin{align}\label{Ishii}
\begin{cases}
 \displaystyle \rho u(x_{n,\epsilon})-\mathcal{L}^{a_{n,\epsilon}}\psi_{n, \epsilon}(x_{n, \epsilon}, \tilde{x}_{n, \epsilon})-\mathcal{I}u(x_{n, \epsilon})\leq 0,\\[0.4em]
 \displaystyle \rho v(\tilde{x}_{n, \epsilon})-\mathcal{L}^{b_{n,\epsilon}}(-\psi_{n, \epsilon})(x_{n, \epsilon},\tilde{x}_{n, \epsilon})-\mathcal{I}v(\tilde{x}_{n, \epsilon})\geq 0,
\end{cases}
\end{align}
where the operator $\mathcal{L}^i$ for $i\in \{a_{n,\epsilon}, b_{n,\epsilon}\}$ acting on $\phi\in C^1(\R)$ is defined by
\begin{align*}
\mathcal{L}^i\phi(x):=\frac{1}{2}\sigma^2i-\delta x\phi'(x)-\Phi(x)+\lambda^{\rm c}(x)H^{({\rm c})}(x;p),~~x\in\R,    
\end{align*}
and $(a_{n,\epsilon}, b_{n,\epsilon})$ satisfy that
\begin{align*}
\begin{pmatrix}
a_{n,\epsilon} & 0\\ 0 &-b_{n,\epsilon}
\end{pmatrix}\leq D_{x,\tilde{x}}^2\psi_{n,\epsilon}(x_{n, \epsilon}, \tilde{x}_{n, \epsilon})+\tilde{\eta} (D_{x,\tilde{x}}^2\psi_{n,\epsilon}(x_{n, \epsilon}, \tilde{x}_{n, \epsilon}))^2.
\end{align*}
Here, $D_{x,\tilde{x}}^2\psi_{n,\epsilon}(x_{n, \epsilon}, \tilde{x}_{n, \epsilon})$ denotes the Hessian matrix of $\psi_{n, \epsilon}$ at the point $(x_{n,\epsilon}, \tilde{x}_{n, \epsilon})$ (and its square is denoted by $(D_{x,\tilde{x}}^2\psi_{n,\epsilon}(x_{n, \epsilon}, \tilde{x}_{n, \epsilon}))^2$), given by
\begin{align*}
D_{x,\tilde{x}}^2\psi_{n,\epsilon}(x_{n, \epsilon}, \tilde{x}_{n, \epsilon})&=\begin{pmatrix}
2n+\epsilon d_{n,\epsilon} &-2n\\
-2n&2n+\epsilon \tilde{d}_{n,\epsilon} 
\end{pmatrix}
\end{align*}
with $d_{n,\epsilon}:=2m(2m-1)|x_{n, \epsilon}|^{2m-2}$ and $\tilde{d}_{n,\epsilon}:=2m(2m-1)|\tilde{x}_{n, \epsilon}|^{2m-2}$.
Taking $\tilde{\eta}=\frac{1}{4n}$. Then, we have
\begin{align*}
& D_{x,\tilde{x}}^2\psi_{n,\epsilon}(x_{n, \epsilon}, \tilde{x}_{n, \epsilon})+\tilde{\eta} (D_{x,\tilde{x}}^2\psi_{n,\epsilon}(x_{n, \epsilon}, \tilde{x}_{n, \epsilon}))^2=4n\begin{pmatrix}
	1&-1\\-1&1
\end{pmatrix}\\
&\qquad\qquad\qquad+\begin{pmatrix}
2\epsilon d_{n,\epsilon} +\frac{(\epsilon d_{n,\epsilon} )^2}{4n} &-\frac{\epsilon }{2}(d_{n,\epsilon} +\tilde{d}_{n,\epsilon} )\\
-\frac{\epsilon }{2}(d_{n,\epsilon} +\tilde{d}_{n,\epsilon} )&2\epsilon \tilde{d}_{n,\epsilon} +\frac{(\epsilon \tilde{d}_{n,\epsilon} )^2}{4n}
\end{pmatrix}.
\end{align*}
Hence, it holds that
\begin{align*}
a_{n,\epsilon}-b_{n,\epsilon}=(1,1)\begin{pmatrix}
	a_{n,\epsilon}&0\\0&-b_{n,\epsilon}
\end{pmatrix}\begin{pmatrix}
1\\1
\end{pmatrix}\leq \epsilon (d_{n,\epsilon} +\tilde{d}_{n,\epsilon} )+\frac{(\epsilon d_{n,\epsilon} )^2+(\epsilon \tilde{d}_{n,\epsilon} )^2}{4n}.    
\end{align*}
By subtracting the two inequalities in \eqref{Ishii}, we obtain 
\begin{align*}
&\rho (u(x_{n, \epsilon})-v(\tilde{x}_{n, \epsilon}))\leq \mathcal{L}^{a_{n,\epsilon}}\psi_{n, \epsilon}(x_{n, \epsilon}, \tilde{x}_{n, \epsilon})-\mathcal{L}^{b_{n,\epsilon}}(-\psi_{n, \epsilon})(x_{n, \epsilon},\tilde{x}_{n, \epsilon})+\mathcal{I}u(x_{n, \epsilon})-\mathcal{I}v(\tilde{x}_{n, \epsilon})\\
&\quad=\frac{1}{2}\sigma^2a_{n,\epsilon}-2nx_{n, \epsilon}\delta(x_{n, \epsilon}-\tilde{x}_{n, \epsilon})-2m\delta\epsilon|x_{n, \epsilon}|^{2m}-\Phi(x_{n, \epsilon})+\lambda^{\rm c}(x_{n, \epsilon})H^{({\rm c})}\left(x_{n, \epsilon}; q^*\right)\\
&\qquad- \frac{1}{2}\sigma^2b_{n,\epsilon}+2n\tilde{x}_{n, \epsilon}\delta (x_{n, \epsilon}-\tilde{x}_{n, \epsilon})+2m\delta\epsilon|\tilde{x}_{n, \epsilon}|^{2m}+\Phi(\tilde{x}_{n, \epsilon})-\lambda^{\rm c}(\tilde{x}_{n, \epsilon})H^{({\rm c})}\left(\tilde{x}_{n, \epsilon};q^*\right)\\
&\qquad+\mathcal{I}u(x_{n, \epsilon})-\mathcal{I}v(\tilde{x}_{n, \epsilon}),
\end{align*}
and there exists a constant $C>0$ such that
{\small\begin{align*}
&\mathcal{I}u(x_{n, \epsilon})-\mathcal{I}v(\tilde{x}_{n, \epsilon})=\lambda^{\rm dp}(x_{n, \epsilon})\sup_{p\in \R_+}\left\{\mu^{\rm dp}(p)\int_0^{\infty}(u(x_{n, \epsilon}+y)-u(x_{n, \epsilon})-py)\nu(dy)\right\}\\
&\qquad\qquad-\lambda^{\rm dp}(\tilde{x}_{n, \epsilon})\sup_{p\in \R_+}\left\{\mu^{\rm dp}(p)\int_0^{\infty}(v(\tilde{x}_{n, \epsilon}+y)-v(\tilde{x}_{n, \epsilon})-py)\nu(dy)\right\}\\
&\leq |\lambda^{\rm dp}(x_{n, \epsilon})-\lambda^{\rm dp}(\tilde{x}_{n, \epsilon})|\left|H_u^{\rm (dp)}(x_{n, \epsilon}; p^*)\right|+\lambda^{\rm dp}(\tilde{x}_{n, \epsilon})\bigg|\int_0^{\infty}(u(x_{n, \epsilon}+y)-v(\tilde{x}_{n, \epsilon}+y))\nu(dy)\nonumber\\
&\qquad-(u(x_{n, \epsilon})-v(\tilde{x}_{n, \epsilon}))\bigg|\\
&\leq C |x_{n, \epsilon}-\tilde{x}_{n, \epsilon}|+\|\lambda^{\rm dp}\|_{\infty}\left( \int_0^{\infty}(M_{n, \epsilon}+\psi_{n, \epsilon}(x_{n,\epsilon}+y, \tilde{x}_{n,\epsilon}+y))\nu(dy)-M_{n, \epsilon}-\psi_{n, \epsilon}(x_{n,\epsilon}, \tilde{x}_{n,\epsilon})\right)\\
&\leq C |x_{n, \epsilon}-\tilde{x}_{n, \epsilon}|+\|\lambda^{\rm dp}\|_{\infty}\left(\epsilon \int_0^{\infty}(|x_{n, \epsilon}+y|^{2m}+|\tilde{x}_{n, \epsilon}+y|^{2m})\nu(dy)-\epsilon  (|x_{n, \epsilon}|^{2m}+|\tilde{x}_{n, \epsilon}|^{2m})\right)\nonumber\\
&\leq C\left(|x_{n, \epsilon}-\tilde{x}_{n, \epsilon}|+\epsilon\right).
\end{align*}}
Letting $n\to \infty$, we obtain from the continuity of $\Phi$ and $H^{(\rm c)}$ and the fact $x_{n, \epsilon}, \tilde{x}_{n, \epsilon} \to x_{\epsilon }$  that $ \rho (u(\bar{x})-v(\bar{x}))-2\epsilon|\bar{x}|^{2m}\leq \rho (u(x_{\epsilon})-v(x_{\epsilon}) )\leq C\epsilon $.  Make $\epsilon \to 0$, we conclude that $0<2\rho\kappa\leq \rho (u(\bar{x})-v(\bar{x}))\leq 0$ , which yields the contradiction. Thus, we complete the proof of the lemma.
\end{proof}

Then,  we have the following main result of this section. 
\begin{theorem}\label{unique}
Let Assumption \ref{ASS} hold. The value function $x\to V(x)$ defined by \eqref{valuefunction2} is the unique viscosity solution to the HJB equation \eqref{eq:HJB} satisfying the polynomial growth condition. Furthermore, the value function $x\to V(x)$ is continuous.
\end{theorem}

\begin{proof}
Let both $x\to u(x)$ and $x\to v(x)$ are viscosity solutions satisfying the polynomial growth condition to the HJB equation \eqref{eq:HJB}.
Recall that $u_*$ and $u^*$ are respectively the u.s.c. envelopes of $u$ and $v$. Then, we have from Lemma \ref{comparison1} that $u^*\leq v_*$ and $v^*\leq u_*$ on $\R$. However, we already have $u_*\leq u\leq u^*$ and $v_*\leq v\leq v^*$ on $\R$. Hence, one has $u_*=u=u^*=v_*=v=v^*$ on $\R$. This proves the uniqueness of the viscosity solution to the HJB equation \eqref{eq:HJB}, and the value function $x\to V(x)$ is both l.s.c. and u.s.c, hence continuous.
\end{proof}

\section{Classical Solution and Verification Result}\label{sec:verification}

This section addresses the well-posedness of the HJB equation \eqref{eq:HJB} in the classical sense. As stated in Theorem \ref{unique}, the value function $x\to V(x)$ given by \eqref{valuefunction2} is the unique viscosity solution of Eq. \eqref{eq:HJB}. However, deriving the optimal pricing policy and validating the verification theorem demands enhanced regularity of the solution. For this reason, we seek to improve the regularity of $x\to V(x)$,  upgrading its characterization from a viscosity solution to a classical $C^2$ solution.

A key observation is that the nonlocal integral term in \eqref{eq:HJB} couples to the value function in a particular manner. If we regard this term as an inhomogeneous source term depending explicitly on $x\to V(x)$, the original integro-differential equation simplifies to a second-order ODE. More precisely, we introduce the following equation, which is defined on $\R$,
\begin{align}\label{eq:ODE}
&\rho u(x)-\frac{1}{2}\sigma^2u''(x)+\delta xu'(x)+\Phi(x)-\lambda^{\rm c}(x)\sup_{q\in \R_+}\left\{\left(1-\mu^{\rm c}\left(\frac{q}{g(x)}\right)\right)q\right\}\nonumber\\
&\qquad-\lambda^{\rm dp}(x)\sup_{p\in \R_+}\left\{\mu^{\rm dp}(p)\int_0^{\infty}(V(x+y)-V(x)-py)\nu(dy)\right\}=0.
\end{align}
Given the value function $x\to V(x)$ defined in \eqref{valuefunction2}, Eq. \eqref{eq:ODE} belongs to a family of nondegenerate HJB equations free of nonlocal integral terms. To establish a connection between \eqref{eq:ODE} and \eqref{eq:HJB}, we first introduce an equivalent definition.
\begin{definition}\label{def2}
\begin{enumerate}
\item[{\rm(i)}] A lower semi‑continuous (l.s.c) function  $x\to v(x)$ is a viscosity supersolution to Eq.~\eqref{eq:HJB} if, for any $\bar{x}\in\R$ and test function $\varphi\in C^2(\R)$ such that $v-\varphi$ attains a minimum at $\bar{x}$ in the sense $(v-\varphi)(\bar{x})=\min_{x\in \R}(v-\varphi)(x)$, it holds that $\rho v(\bar{x})-\mathcal{L}\varphi(\bar{x})-\mathcal{I}v(\bar{x})\geq 0$. 
\item[{\rm(ii)}] An upper semi‑continuous (u.s.c) function $u$ is a viscosity subsolution to \eqref{eq:HJB} if, for any $\bar{x}\in\R$ and test function $\varphi\in C^2(\R)$ such that $u-\varphi$ attains a maximum at $\bar{x}$ in the sense $(u-\varphi)(\bar{x})=\max_{x\in \R}(u-\varphi)(x)$, it holds that $\rho u(\bar{x})-\mathcal{L}\varphi(\bar{x})-\mathcal{I}u(\bar{x})\leq 0$.
\item[{\rm(iii)}] A locally bounded function $x\to u(x)$ is a viscosity solution of \eqref{eq:HJB} if its u.s.c. envelope $x\to u^*(x)$ and its l.s.c. envelope $x\to u_*(x)$ are respectively a viscosity subsolution and a viscosity supersolution to \eqref{eq:HJB}.
\end{enumerate}
\end{definition}
The equivalence between Definition \ref{def1} and Definition \ref{def2} is standard in the viscosity solution theory for integro-differential equations, see, e.g. \cite{BI08}. 
It enable us to investigate the solvability of Eq. \eqref{eq:ODE}. From Definition \ref{def2}, we can see that the value function $V$ defined by \eqref{valuefunction2} being a viscosity solution to Eq. \eqref{eq:HJB}, is automatically a viscosity solution to Eq. \eqref{eq:ODE}. Consequently, If Eq. \eqref{eq:ODE} admits a unique viscosity solution $u$, then $u = V$. Moreover, $u\in C^2(\R)$ implies $V\in C^2(\R)$.

The following result provides the well-posedness of classical solutions to Eq. \eqref{eq:ODE} and \eqref{eq:HJB}. 
\begin{theorem}\label{classical}
Let Assumption \ref{ASS} hold. Then, the value function $V$ defined by \eqref{valuefunction2} is the only classical solution to Eq. \eqref{eq:ODE} and Eq. \eqref{eq:HJB} satisfying the polynomial growth condition.
\end{theorem}

\begin{proof}
The proof of uniqueness of viscosity solution to Eq. \eqref{eq:ODE} is in line with that in  Lemma \ref{comparison1} respectively, so we omit them here. Based on these results, we know that value function $x\to V(x)$ defined by \eqref{valuefunction2} is the unique viscosity solution to \eqref{eq:ODE}. To elevate this viscosity solution to a classical one, let us consider a localized Dirichlet problem. For arbitrary $x_1,x_2\in\R$ with $x_1\leq x_2$, we consider 
\begin{align}\label{dirichlet}
\rho u(x)-\mathcal{L}u(x)-\mathcal{I}V(x)=0,\  u(x_1)=V(x_1), \ u(x_2)=V(x_2).
\end{align}
It follows from Proposition D.6 in \cite{HL12} under Assumption \ref{ASS} that $\mathcal{I}V(x)$ is continuous in $x\in \R$. Then, classical results for linear elliptic PDEs (and, in one dimension, ODEs) guarantee the existence and uniqueness of a classical solution $u \in C^2(x_1, x_2) \cap C([x_1, x_2])$ to Eq.~\eqref{dirichlet} (see, e.g., Theorem 6.8 of \cite{GT77} or Theorem 6.2.4 of \cite{F75}). By definition, this classical solution $u$ is automatically a viscosity solution to \eqref{dirichlet} on the bounded domain $(x_1, x_2)$.  By virtue of comparison principles (\cite{P09})  on bounded domains, it follows that $u = V$ on $(x_1, x_2)$. Since $x_1,x_2\in\R$ are arbitrary, we conclude that $V\in C^2(\R)$, which implies $x\to V(x)$ is indeed a classical solution to Eq. \eqref{eq:ODE} and Eq. \eqref{eq:HJB}. 
\end{proof}

A verification result is provided to ensure that the classical solution to the HJB equation \eqref{eq:HJB} is indeed the value function.
\begin{theorem}[Verification result]
Let $v\in C^2(\R)$ be the classical solution to Eq.  \eqref{eq:HJB}. Consider the optimal (feedback) control functions $x\to p^*(x)$ and $x\to q^*(x)$ given by \eqref{controlfuncition} and the state process $X^*=(X_t^*)_{t\geq 0}$ satisfying the dynamics \eqref{state2} with $(p,q)=(p_t,q_t)_{t\geq 0}$ replaced by $(p^*(X^*), q^*(X^*))=(p^*(X_t^*), q^*(X_t^*))_{t\geq0}$. Then, we have, for any $x\in\R$, 
\begin{align*}
v(x)=V(x)=J(x; p^*, q^*)=\sup_{(p, q)\in {\cal U}}J(x;p,q),    
\end{align*}
and $(p^*(X^*), q^*(X^*))\in{\cal U}$ is an optimal pricing strategy.
\end{theorem}

\begin{proof}
For any $(p,q)=(p_t,q_t)_{t\geq 0}\in {\cal U}$, let $X=(X_t)_{t\geq 0}$ be the state process satisfying the dynamics \eqref{state2} under $(p,q)\in {\cal U}$ and $v\in C^2(\R)$ be a classical solution to \eqref{eq:HJB}. For $n\geq1$, define $\tau_n=\inf\{t\geq 0;~\int_0^t|v'(X_s)\sigma|^2ds\geq n\}$ with the convention $\inf\varnothing=+\infty$. For any $T>0$, by applying It\^o's formula to $e^{-\rho t}v(X_t)$ from $t=0$ and $T\wedge \tau_n$, we have 
\begin{align*}
&e^{-\rho (T\wedge \tau_n)}v(X_{T\wedge \tau_n})=v(x)+\int_0^{T\wedge \tau_n}e^{-\rho t}\left(-\rho v(X_t)+\frac{\sigma^2}{2}v''(X_t)-\delta X_tv'(X_t)\right)dt\\
&\quad+\int_0^{T\wedge \tau_n}\int_0^{\infty}\int_0^{\infty}e^{-\rho t}1_{\{z\leq \lambda^{\rm dp}(X_{t-})\mu^{\rm dp}(p_t)\}}\left(v(X_{t-}+y)-v(X_{t-})\right)\mathcal{N}(dt,dy,dz)\\
&\quad+\int_0^{T\wedge \tau_n}e^{-\rho t}v'(X_t)\sigma dW_t.
\end{align*}
Taking expectations on both sides of the above display, we obtain
\begin{align*}
&\mathbb{E}_x[e^{-\rho {T\wedge \tau_n}}v(X_{T\wedge \tau_n})]=v(x)+\mathbb{E}_x\left[\int_0^{T\wedge \tau_n}e^{-\rho t}\left(-\rho v(X_t)+\frac{\sigma^2}{2}v''(X_t)-\delta X_tv'(X_t)\right)dt\right]\\
&\quad+\mathbb{E}_x\left[\int_0^{T\wedge \tau_n}\int_0^{\infty}e^{-\rho t}\left(v(X_{t-}+y)-v(X_{t-})\right)\lambda^{\rm dp}(X_{t-})\mu(p_t)\nu(dy)dt\right].
\end{align*}
By the polynomial growth condition of the classical solution $x\to v(x)$ to \eqref{eq:HJB}, one may apply the dominated convergence theorem (DCT) and tend $n$ to infinity. By using \eqref{eq:HJB}, it holds that
{\small\begin{align}\label{eq:veryine0}
&\mathbb{E}_x\left[e^{-\rho T}v(X_T)\right]\leq v(x)\\
&~-\mathbb{E}_x\left[\int_0^Te^{-\rho t}\left(\lambda^{\rm c}(X_t)\left(1-\mu^{\rm c}\left(\frac{q_t}{g(X_t)}\right)\right)q_t-\lambda^{\rm dp}(X_t)\mu^{\rm dp}(p_t)p_t\int_0^{\infty}y\nu(dy)-\Phi(X_t)\right)dt\right].\nonumber
\end{align}}
Using the growth condition (with growth constant $M$) satisfied by $x\to v(x)$ again, there exists a constant $C>0$ such that
\begin{align*}
\mathbb{E}_x\left[e^{-\rho T}|v(X_T)|\right]\leq Me^{-\rho T}\left\{1+\mathbb{E}_x\left[|X_T|^m\right]\right\}\leq Me^{-\rho T}+Ce^{-(\rho-\rho_0) T}\left(1+|x|^m\right).
\end{align*}
Hence, for $\rho >\rho_0$, we arrive at $\lim_{T\to \infty}\mathbb{E}[e^{-\rho T}v(X_T)]=0$. This together with \eqref{eq:veryine0}, letting $T\to\infty$, we have from DCT that, for any $x\in \R$,
\begin{align}\label{ve1}
&\mathbb{E}_x\left[\int_0^{\infty}e^{-\rho t}\left(\lambda^{\rm c}(X_t)\left(1-\mu^{\rm c}\left(\frac{q_t}{g(X_t)}\right)\right)q_t-\lambda^{\rm dp}(X_t)\mu^{\rm dp}(p_t)p_t\int_0^{\infty}y\nu(dy)-\Phi(X_t)\right)dt\right]\nonumber\\
&\qquad\leq v(x).
\end{align}
By the arbitrariness of $(p,q)$, we have $V(x)\leq v(x)$ for all $x\in \R$. On the other hand,
the equality in \eqref{ve1} holds when $(p,q)=(p^*,q^*)$, which implies $V(x)\geq v(x)$ for $x\in \R$. Then, we have $v=V$ on $\R$. 

Next, we show $(p^*,q^*)\in \mathcal{U}$. The integrability of $p^*$ can be deduced from the fact that 
\begin{align*}
p^*(x)\leq |z(x)|:=\left| \frac{\int_0^{\infty}(V(x+y)-V(x))\nu(dy)}{\int_0^{\infty}y\nu(dy)}\right|,\quad \forall x\in\R,     
\end{align*}
and the polynomial condition satisfied by $x\to V(x)$.
The control $q^*$ is also integrable as it admits the feedback control form $q^*(x)=\hat{q}g(x)$ and $x\to g(x)$ is bounded on $\R$.
It suffices to show that there exists a unique solution to Eq. \eqref{state2} under $(p^*, q^*)$. 
By the locally Lipschitz property of $p^*(x)$ obtained in Lemma \ref{control}, we have that for any compact set $K\subset \R$, there exists a positive constant $L_K>0$ such that
\begin{align*}
|p(x_1)-p(x_2)|\leq L_K|x_1-x_2|, \quad \forall x_1, x_2 \in K.  
\end{align*}
By using Assumption~\ref{ASS}, $p\to\mu^{\rm dp}(p)$ is continuous on $\R_+$ and $C^1$ on $(0, b)$ with $\mu^{\rm dp}(p) = 1$ for all $p \ge b$. Under Assumption \ref{ASS}, $p\to\mu^{\rm dp}(p)$ has bounded one-sided derivatives at both boundaries $0$ and $b$, ensuring its Lipschitz continuity on any compact subset of $\R_+$.
Therefore, there exists a constant $M_{K'} > 0$ such that
\begin{align*}
|\mu^{\rm dp}(p_2) - \mu^{\rm dp}(p_1)| \le M_{K'} |p_2 - p_1|, \quad \forall p_1, p_2 \in K':=\{p^*(x):~x\in K\}.    
\end{align*}
For any $x_1, x_2 \in K$, we  obtain
\begin{align*}
|\mu^{\rm dp}(p^*(x_2)) - \mu^{\rm dp}(p^*(x_1))| \le M_{K'} |p^*(x_2) - p^*(x_1)| \le M_{K'} L_K |x_2 - x_1|.   
\end{align*}
Defining $L_{\mu} := M_{K'} L_K < +\infty$, we have
\begin{align*}
|\mu^{\rm dp}(p^*(x_2)) - \mu^{\rm dp}(p^*(x_1))| \le L_{\mu} |x_2 - x_1|, \quad \forall x_1, x_2 \in K.    
\end{align*}
Hence, the mapping $x \mapsto \mu^{\rm dp}(p^*(x))$ is locally Lipschitz continuous on $\R$. Moreover, $x\to \lambda^{dp}(x)\mu^{\rm dp}(p^*(x))$ is bounded on $\R$. Following a standard argument as in Theorem 5.2.2 of \cite{F75}, we can obtain the desired result.
\end{proof}

\section{Numerical Examples}\label{sec:example}

This section implements a numerical experiment to demonstrate the properties of the optimal pricing strategies and value function derived earlier. To acquire tangible numerical results, we utilize canonical functional forms prevalent in data pricing research, complying with all the model assumptions established above.

Following \cite{ZP17}, we consider the case where both the providers’ privacy valuations and the consumers’ willingness‑to‑pay follow exponential distributions. In particular, $\mu^{\rm dp}(x)=1-e^{- \alpha_1 x}$ and $\mu^{\rm c}(x)=1-e^{- \alpha_2x}$ for $x\in \R$. Here, the intensity parameters $\alpha_1, \alpha_2>0$. For the cost function which describes the platform’s operational expenses (e.g., storage, maintenance and data processing), we take the following quadratic cost function given by $\Phi(x) = \phi_0+\phi_1x+\phi_2x^2$ for $x\in\R$. Here, the parameters $\phi_0,\phi_1,\phi_2>0$. This convex form reflects the empirically observed feature of increasing marginal costs: as the data volume grows, each additional unit of data becomes more expensive to store and process.
The intensities of data providers and consumers are assumed to depend linearly on the positive part of the data volume in the sense that
\begin{align*}
\lambda^{\rm dp}(x)=\min\{M^{\rm dp}, L^{\rm dp}x^++l^{\rm dp}\},\quad \lambda^{\rm c}(x)= L^{\rm c}x^++l^{\rm c},
\end{align*}
where, $x^+:=\max\{x,0\}$ for $x\in\R$ and the parameters $M^{\rm dp}, L^{\rm dp}, L^{\rm c}, l^{\rm dp}, l^{\rm c}$ are positive constants. The upper bound $M^{\rm dp}$ prevents the provider arrival rate from growing without bound, which is realistic given finite market size. The quality of the data product, which influences consumers’ willingness‑to‑pay, is modeled as the following saturating function given by
\begin{align}\label{eq:numerg}
g(x)=\begin{cases}
\displaystyle \qquad\qquad\qquad\epsilon \Lambda, & x\leq 0,\\
\displaystyle \epsilon\Lambda+(1-\epsilon)\Lambda\left(1-\frac{1}{1+k_1x}\right), & x>0,
\end{cases}     
\end{align}
where $\epsilon\in (0,1)$, the parameter $\Lambda>0$ is the maximum achievable quality (e.g., accuracy for machine learning models) and $k_1>0$ are curve fitting parameters. The saturating function $g:\R\to(0,\infty)$ is non-decreasing, which has the decreasing marginal quality and it increases asymptotically towards $\Lambda$ as the data size grows. A slight modification is made to the quality function proposed by \cite{SK22}. 
\begin{figure}[h]
    \centering
    \includegraphics[width=0.6\textwidth]{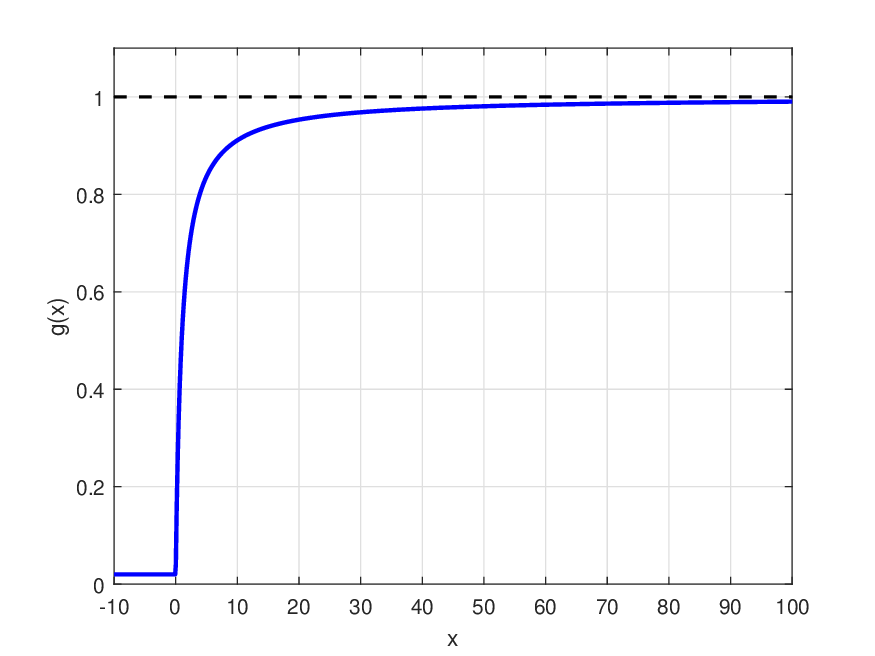}
    \caption{The curve of the saturating function $x\to g(x)$ defined by \eqref{eq:numerg}.}
    \label{fig:g}
\end{figure} 
From Lemma \ref{control}, there exists a unique pair of optimal control $(p^*,q^*)\in{\cal U}$. For the value function $x\to V(x)$ defined by \eqref{valuefunction2}, i.e., $V(x)=J(x;p^*,q^*)$ for $x\in\R$, let us define $\Delta_V(x):=\int_{0}^{\infty}\left(V(x+y)-V(x)\right)\nu(dy)$.
Under the above setting, we have from Lemma \ref{control} that, the optimal provider's price is given by, for $x\in\R$, 
\begin{align*}
p^*(x)=\begin{cases}
\displaystyle ~~0, & \Delta_V(x)\leq 0,\\
\displaystyle \hat{p}(x), & \Delta_V(x)>0,
\end{cases}
\end{align*}
where, for any $x\in\R$, the quantity $\hat{p}(x)$ solves the nonlinear equation in the unknown variable $z\in\R_+$ given by
\begin{align}\label{eq:equhatp0}
e^{ \alpha_1 z}+ \alpha_1z=\frac{ \alpha_1}{\int_{0}^{\infty}y\nu(dy)} \Delta_V(x)+1.    
\end{align}
In fact, the function $z\to h(z):=e^{\alpha_1 z}+\alpha_1 z$ is strictly increasing on $\R_+$, which satisfies $h(0)=1$ and $\lim_{z\to\infty}h(z)=+\infty$. Thus, for $\Delta_V(x)\geq 0$, there exists a unique solution $z\in \R_+$ to Eq. \eqref{eq:equhatp0}. For $\Delta<0$, no positive solution exists, and the optimal provider price is given by the boundary value $0$.  For the optimal sale price $x\to q^*(x)$, we have $q^*(x)=\frac{g(x)}{ \alpha_2}$ for $x\in\R$. In view of \eqref{eq:HJB}, the value function $x\to V(x)$ solves the following equation:
\begin{align}\label{exampleode}
\rho V(x)+\delta xV'(x)-\frac{\sigma^2}{2}V''(x)+\Phi(x)=\lambda^{\rm c}(x)H^{({\rm c})}(x;q^*)+\lambda^{\rm dp}(x)H_V^{({\rm dp})}(x;p^*),
\end{align}
where, the function $H_V^{({\rm dp})}(x;p^*)$ for $x\in\R$ is defined by
\begin{align*}
H_V^{({\rm dp})}(x;p^*)=\begin{cases}
\displaystyle ~~~~~~~~~~~~~~~~~~0,& \Delta_V (x)\leq 0,\\[0.4em]
\displaystyle \frac{\int_{0}^{\infty}y\nu(dy)}{ \alpha_1} \left(e^{ \alpha_1 \hat{p}(x)}+e^{-\alpha_1 \hat{p}(x)}-2\right), & \Delta_V(x)>0,
\end{cases}    
\end{align*}
and the function $H^{({\rm c})}(x;q^*)$ for $x\in\R$ is defined by $H^{({\rm c})}(x;q^*)=\frac{g(x)}{\alpha_2}e^{-1}$ for $x\in \R$.

Because an analytical solution to the HJB equation \eqref{exampleode} is not available,  we turn to solve it numerically using a finite difference method for analyzing structural properties satisfied by the optimal price strategies. The computational procedure consists of the following steps:
\begin{itemize}
\item[1.]{\bf Domain truncation and grid division:} The unbounded state space $\R$ is truncated to a sufficiently large interval $[x_{\min}, x_{\max}]$. We then divide the interval $[x_{\min}, x_{\max}]$ into $N_x$ segments with the step size $\Delta x=\frac{x_{\max}-x_{\min}}{N_x}$ and  $x_i=x_{\min}+(i-1)\Delta x, i=1, \ldots, N_x$. And we obtain $N_y$, $y_{\max}$ and $ \Delta y$ by a similar argument.
\item[2.]{\bf Discretizations of derivative terms:} The upwind schemes is used to approximate the first-order derivative term. That is, use the forward difference if the velocity at a point is non-negative, and use the backward difference otherwise. A central difference schemes is used to approximate the second-order derivative term.
\item[3.]{\bf Approximation of the integral term:} For the integral term, we use the following form of approximation: 
\begin{align*}
\mathcal{I}V(x_i) \approx \sum_{k=1}^{N_y} (V_{\text{tail}}(x_i + y_k) - V_i) \eta e^{-\eta y_k}\Delta y,    
\end{align*}
where, $V_{\rm tail}$ is set to be $V_{\rm tail}=a_Rx^2+b_Rx+c_R$, and $a_R, b_R, c_R$ are obtained by fitting on the current solution of $x\to V(x)$ near the boundaries.
\item[4.]{\bf Discrete system of equations:} We obtain the discrete version of Eq. \eqref{exampleode} as follows:
\begin{align*}
&\rho V(x_i)-\frac{\sigma^2}{2}\frac{V(x_{i+1})-2V(x_i)+V(x_{i-1})}{\Delta x^2}+\delta x_i^+\frac{V(x_{i+1}-V(x_i))}{\Delta x}\\
&\quad+\delta x_i^-\frac{V(x_i)-V(x_{i-1})}{\Delta x}=\lambda^{\rm dp}(x_i)H_V^{({\rm dp})}(x_i; p^*)+\lambda^{\rm c}(x_i)H^{({\rm c})}(x_i; q^*)-\Phi(x_i).
\end{align*}
For $i=2, \ldots, N_x-1$, denote by $b_i:=\lambda^{\rm dp}(x_i)H_V^{({\rm dp})}(x_i; p^*)+\lambda^{\rm c}(x_i)H^{({\rm c})}(x_i; q^*)-\Phi(x_i)$ and $V_i:=V(x_i)$. For $i=1$ and $i=N_x$, we impose quadratic extrapolation conditions consistent with the quadratic growth of $V$, that is, $V_1:=3V_2-3V_3+V_4$, $V_{N_x}:=3V_{N_x-1}-3V_{N_x-2}+V_{N_x-3}$.
Then, the following sparse linear system is obtained:
\begin{align}\label{linearsystem}
\boldsymbol{A}\boldsymbol{V}=\boldsymbol{b},    
\end{align}
where, the vectors $\boldsymbol{V}=(V_1, \ldots, V_{N_x})^{\top}$ and $\boldsymbol{b}=(b_1, \ldots, b_{N_x})^{\top}$. The matrix coefficient $\boldsymbol{A}=(A_{i,j})_{N_x\times N_x}$ is given by
\begin{align*}
\begin{cases}
\displaystyle A_{1,1}=1,~~ A_{1,2}=-3,~~ A_{1,3}=3,~~ A_{1,4}=-1,\\[0.4em]
\displaystyle A_{N_x,N_x-3}=-1,~~ A_{N_x,N_x-2}=3,~~ A_{N_x,N_x-1}=-3,~~ A_{N_x,N_x}=1,\\[0.4em]
\displaystyle A_{i, i-1}=-\frac{\sigma^2}{2\Delta x^2}-\frac{\delta x_i^-}{\Delta x},~~  A_{i,i}=\rho+\frac{\sigma^2}{\Delta x^2}-\frac{\delta x_i^+}{\Delta x}+\frac{\delta x_i^-}{\Delta x},\\[0.3em] 
\displaystyle A_{i, i+1}=-\frac{\sigma^2}{2\Delta x^2}+\frac{\delta x_i^+}{\Delta x},~~\text{for}~i=2, \ldots,N_x-1.
\end{cases}
\end{align*}

\item[5.]{\bf Iteration with relaxation:} We then adopt the iteration and relaxation technique with the form $V^{k+1}=\theta V^{\rm new}+(1-\theta)V^k$, where $V^{\rm new}$ is obtained by solving the linear system \eqref{linearsystem} and $\theta\in (0,1]$ is the relaxation factor. 
\end{itemize}  

The values of the relevant parameters are provided in Table \ref{para}.
\begin{table}[ht]
	\centering
	\caption{The model parameter values used in the numerical experiment.}
    \label{para}
	\begin{tabular}{@{}lcccccccccccccc@{}}
		\toprule
		\textbf{Parameter}  &  $\alpha_1$ &$\alpha_2$&$k_1$&$\Lambda$ &$\phi_2$& $L^{\rm dp}$ &$L^{\rm c}$ & $M^{\rm dp}$& $l^{\rm dp}$&$l^{\rm c}$ &$\rho$ &$\delta$ &  $\sigma $\\
		\midrule
        \textbf{Value}&0.5& 1.5 &  1.0 &0.8 &0.05 &1.0 &  2.0&1.0&0.1&0.1& 3& 0.1& 0.1\\
		\bottomrule
	\end{tabular}
\end{table}
\begin{figure}[H]
    \centering
    \begin{subfigure}[b]{0.45\textwidth}
        \centering
        \includegraphics[width=\textwidth]{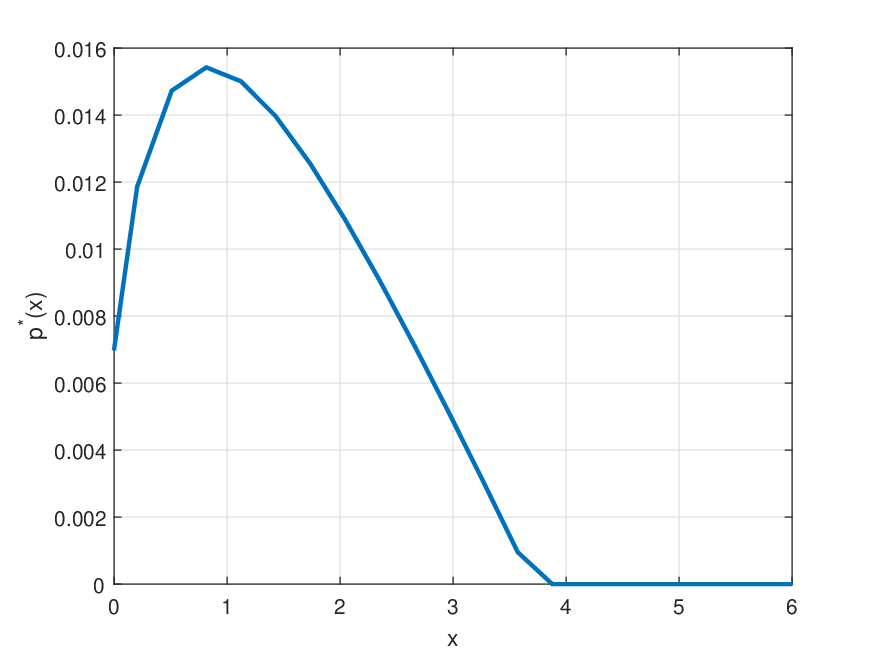}
    \end{subfigure}
    \begin{subfigure}[b]{0.45\textwidth}
        \centering
        \includegraphics[width=\textwidth]{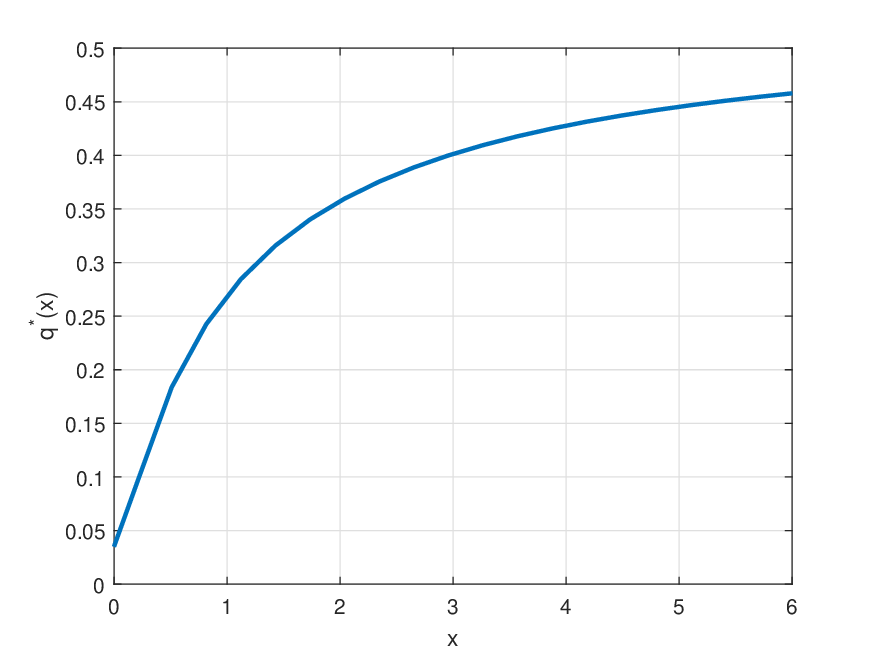}
    \end{subfigure}
    \caption{ Left panel: optimal provider price $x\to p^*(x)$. Right panel: optimal consumer price $ x\to q^*(x)$.}
    \label{fig:pq}
\end{figure}
Figure~\ref{fig:pq} displays the optimal provider price $p^*(x)$ and and the optimal selling price $q^*(x)$ as functions of the current data volume $x$. The optimal consumer price  $q^*(x) $ increases monotonically with respect to the initial volume of data $x$ and eventually converges to a constant, indicating that a larger data volume allows the platform to charge a higher price. This upward trend aligns with the shape of the quality curve.  In contrast, the optimal  provider price $ p^*(x)$ exhibits a non‑monotonic pattern: it first rises from zero, reaches a maximum at an intermediate data volume, and then gradually declines back to zero. 
This non-monotonic pattern mirrors the behavior of the value function $V(x)$. Recall that $\Delta_V(x)$ represents the expected increase in the platform’s profit 
when a batch of random data volume $Y$ is added to the current data volume $x$. 
In other words, $\Delta_V(x)$ measures the marginal benefit of acquiring an 
additional unit of data (in expectation). When $\Delta_V(x)$ is positive, the 
platform expects future profits to rise by accumulating more data, so it is 
willing to pay a positive provider price. However, once $\Delta_V(x)$ becomes 
negative, the expected marginal benefit of additional data no longer outweighs 
the costs. At this point, the platform ceases to pay for further data accumulation 
in order to balance profit and cost effectively, it stops purchasing new data.

In the early stages of platform development, the primary focus is on  database expansion. With only a small volume of data, the platform  cannot yet offer high‑quality products or attract many consumers. Therefore,  it invests heavily in data acquisition, even at a relatively high price, to  build a critical mass. As the platform matures and its database grows, it  becomes better positioned to leverage the accumulated data: it can provide  consumers with refined data products of higher quality, which in turn increases  revenue. This enhanced ability to generate profit from existing data reduces  the need for further accumulation. Consequently, the platform shifts its  strategy from aggressive acquisition to careful cost management, and the  optimal provider price $p^*(x)$ declines. Thus, the inverted‑U shape of $p^*(x)$ reflects a natural life‑cycle transition:  from a growth phase (where data are scarce and valuable) to a  maturity phase (where the database is already large and additional data yield diminishing returns). 

\begin{figure}[ht]
    \centering
    \includegraphics[width=0.6\textwidth]{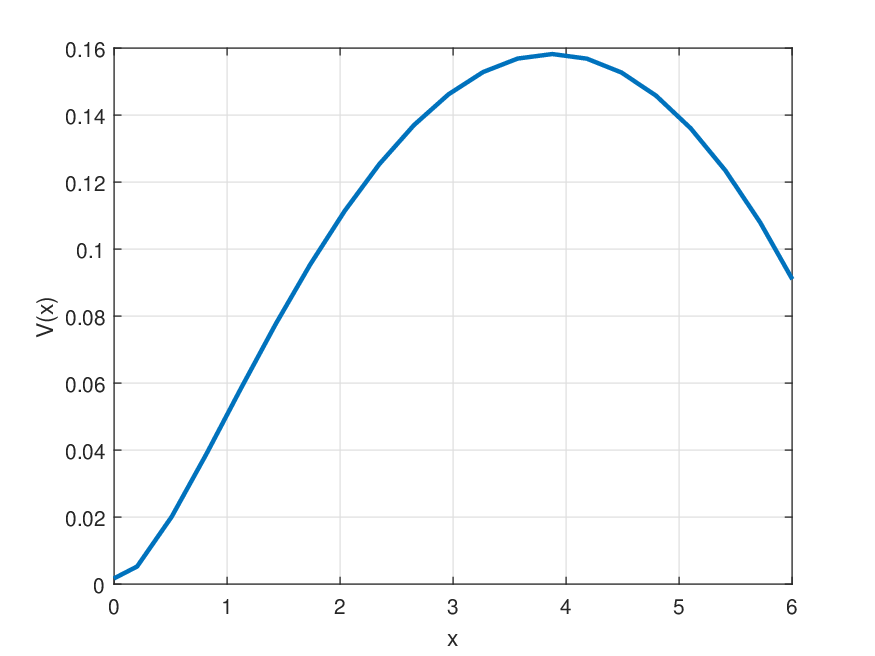}
    \caption{The value function $x\to V(x)$.}
    \label{fig:V}
\end{figure} 
As shown in Figure~\ref{fig:V}, the value function $ V(x) $ exhibits an inverted U-shape: it first increases from a low value, reaches a maximum, and then decreases. This indicates that there exists an optimal data volume level that maximizes the platform's expected profit. For small $ x $, the platform benefits from data accumulation because both the consumer price $ q^*(x) $ and provider price $ p^*(x) $ increase. However, beyond the optimal point, the convex storage cost $ \Phi(x) $ and the saturation of query quality $ g(x) $ dominate, making further data accumulation profitless.
The inverted U-shaped value function reveals a fundamental trade-off in data platform management, which aligns with the results presented in \cite{JW18}. In the initial phase, accumulating data improves both consumer-side revenue (via higher $q^*(x)$) and provider-side contribution (via positive $ p^*(x)$). Consequently, the platform's profit rises. Beyond the optimal stock, the convex storage cost and quality saturation outweigh the benefits: additional data no longer enhance query quality, and the marginal value of data becomes negative, driving $p^*(x)$ down to zero. Consequently, the platform's profit declines, suggesting that excessive data accumulation can be economically inefficient.

We now study the influence of the parameters on the pricing strategies.  We plot in Figure~\ref{fig:k1} the optimal prices and the value function when the parameter $k_1$ varies. We observe that increasing $k_1$ leads to a uniform upward shift in all three curves. The parameter $k_1$ can be interpreted as the sensitivity factor of data volume with respect to quality enhancement. As $k_1$ increases, the quality  $g(x)$ approaches its saturation level more rapidly. This quality improvement makes consumers more willing to pay, allowing the platform to charge a higher price, which is a direct effect of the quality premium. A higher sale price increases both the platform's instantaneous revenue and the expected incremental value of data $\Delta_V(x)$, thereby enhancing the platform's incentive to acquire more data. Consequently, the positive effect of $k_1$ on $q^*(x)$ propagates through the value function to $p^*(x)$ and ultimately to all key variables.
\begin{figure}[H]
    \centering
    \begin{subfigure}[b]{0.3\textwidth}
        \centering
        \includegraphics[width=\textwidth]{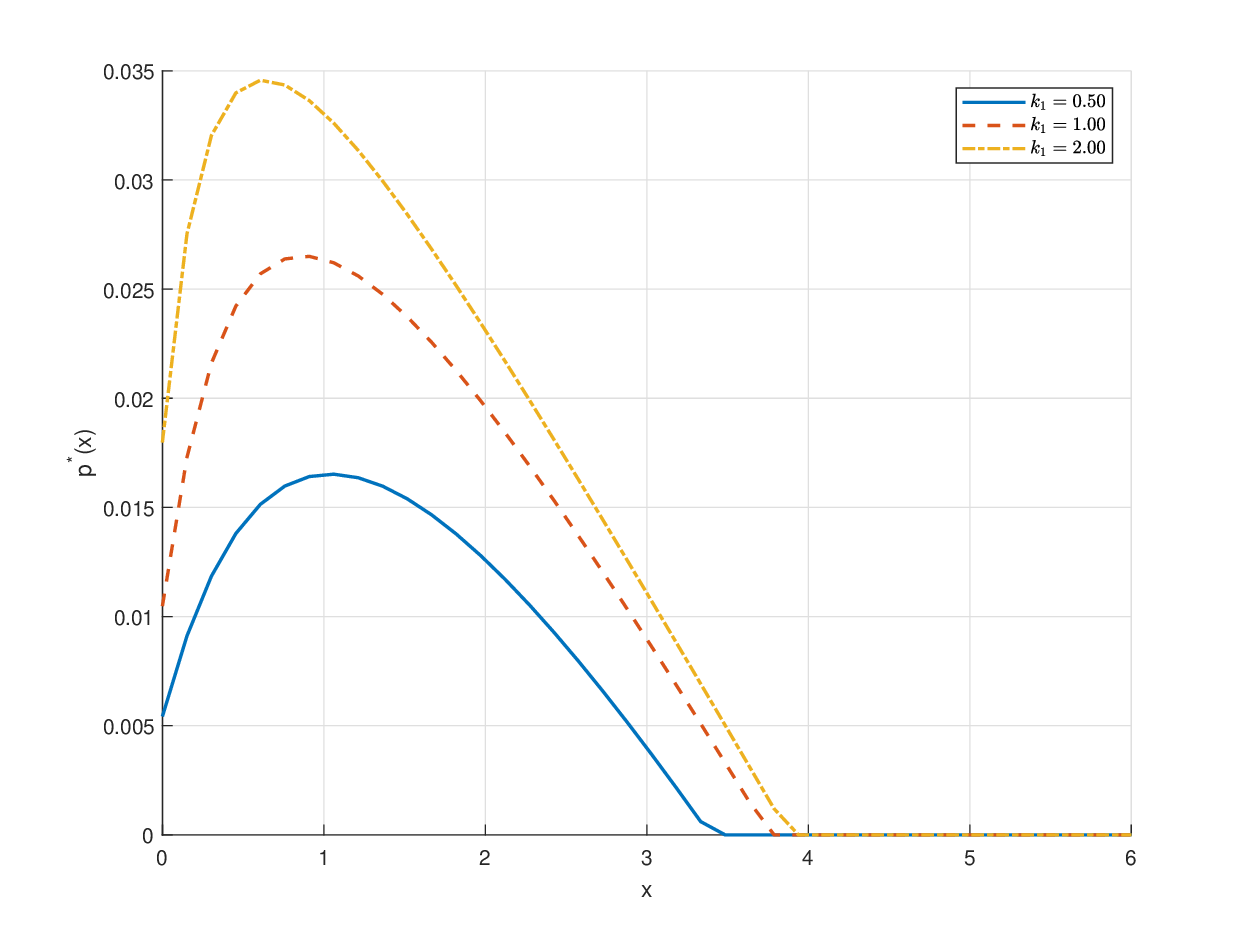}
    \end{subfigure}
    \begin{subfigure}[b]{0.3\textwidth}
        \centering
        \includegraphics[width=\textwidth]{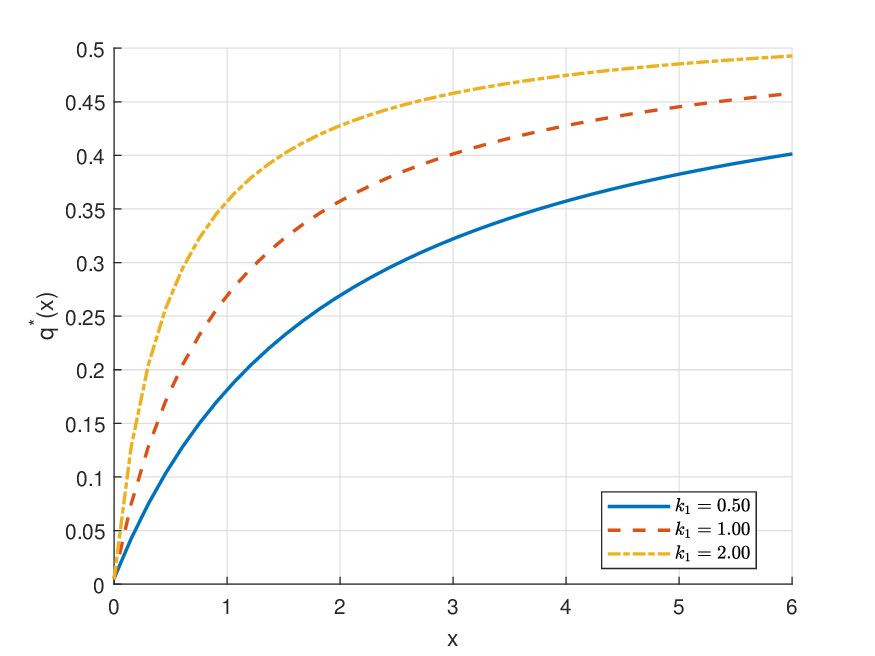}
    \end{subfigure}
    \begin{subfigure}[b]{0.3\textwidth}
        \centering
        \includegraphics[width=\textwidth]{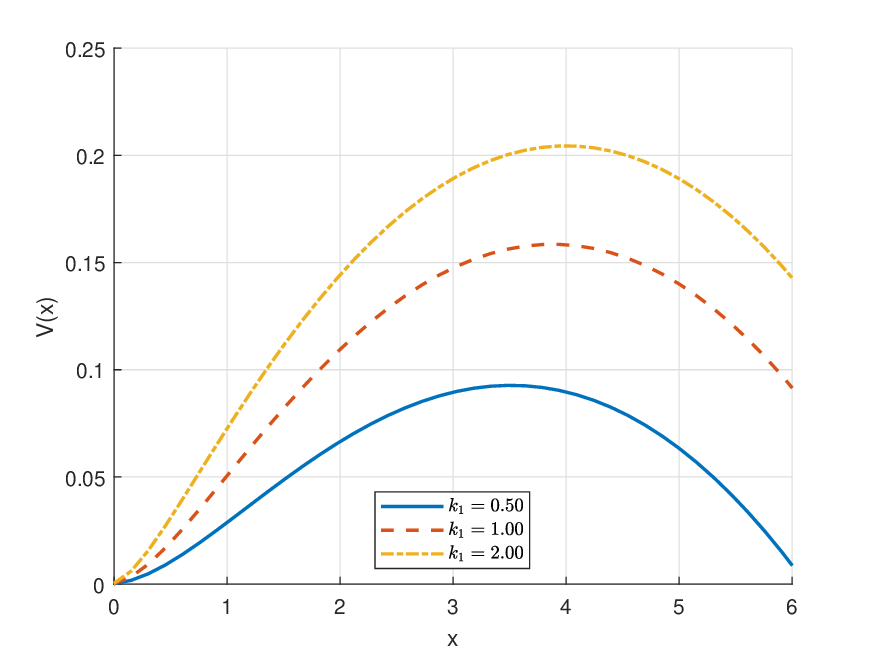}
    \end{subfigure}
    \caption{ Left panel: optimal provider price $x\to p^*(x)$. Middle panel: optimal consumer price $ x\to q^*(x)$. Right panel: value function $x\to V(x)$. The parameter values are $k_1=0.5, 1.0, 2.0$}
    \label{fig:k1}
\end{figure}

\begin{figure}[H]
    \centering
    \begin{subfigure}[b]{0.3\textwidth}
        \centering
        \includegraphics[width=\textwidth]{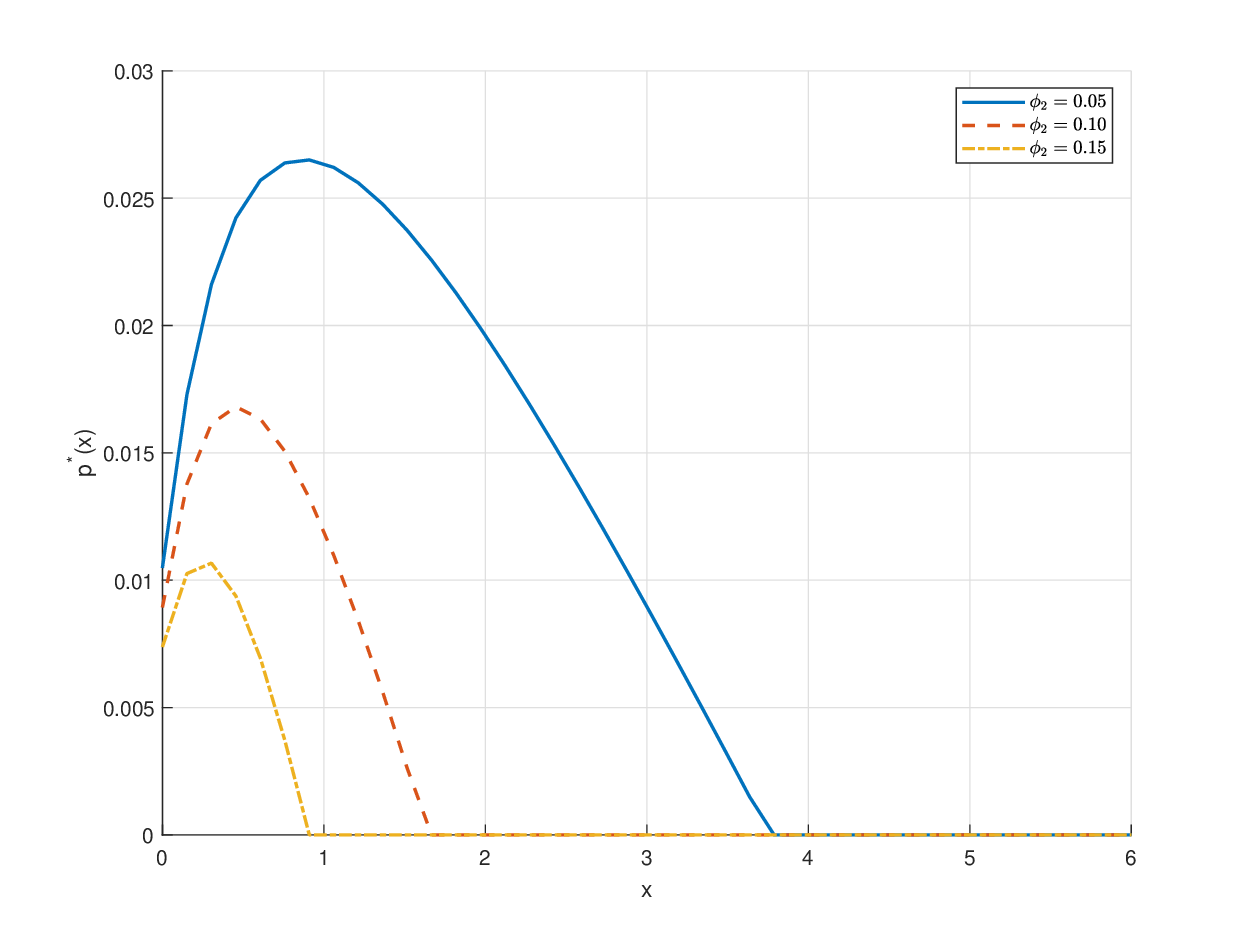}
    \end{subfigure}
    \begin{subfigure}[b]{0.3\textwidth}
        \centering
        \includegraphics[width=\textwidth]{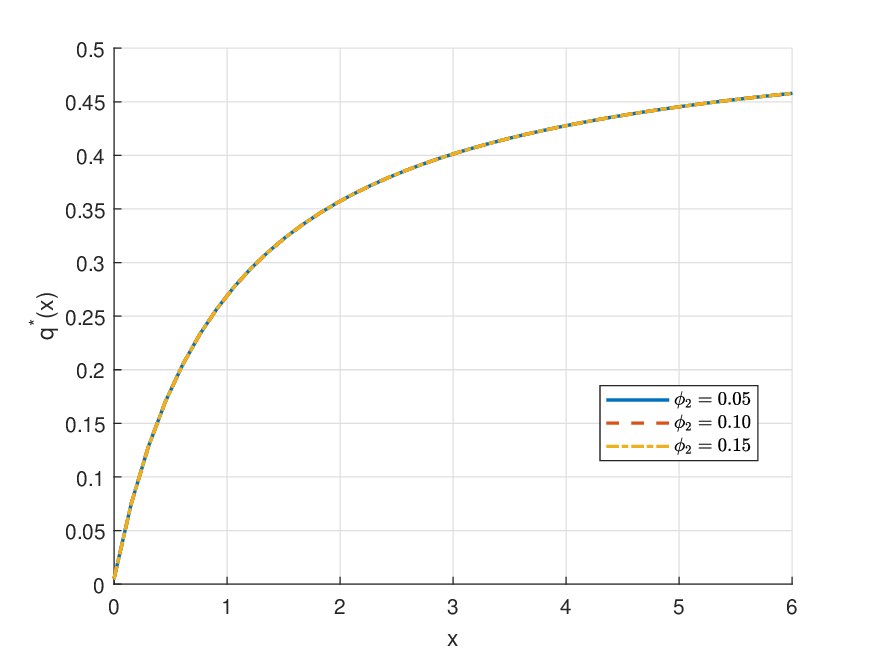}
    \end{subfigure}
    \begin{subfigure}[b]{0.3\textwidth}
        \centering
        \includegraphics[width=\textwidth]{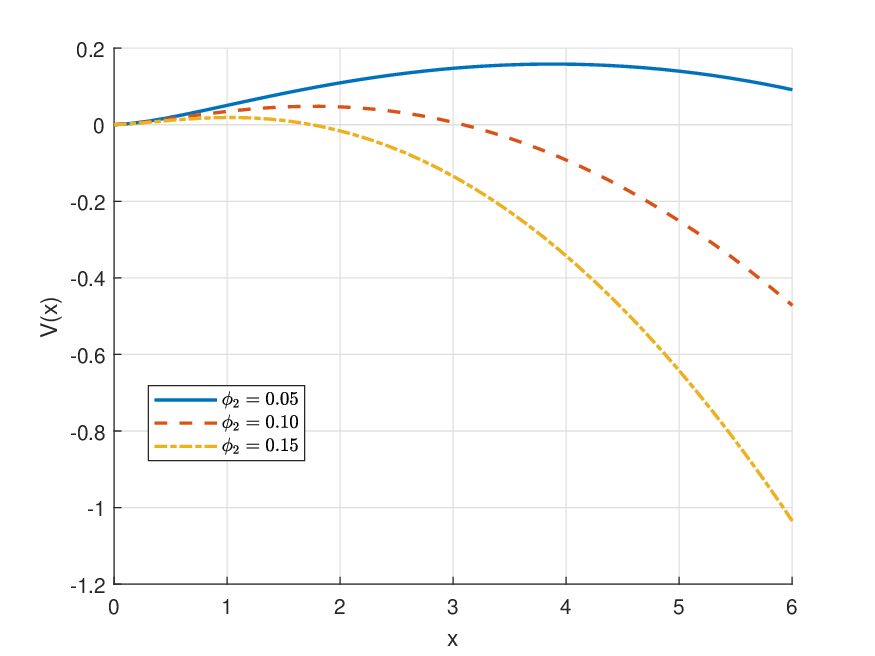}
    \end{subfigure}
    \caption{ Left panel: Optimal provider price $x\to p^*(x)$. Middle panel: optimal consumer price $x\to q^*(x)$. Right panel: value function $ x\to V(x)$. The parameter values are $\phi=0.05, 0.10, 0.15$}
    \label{fig:phi}
\end{figure}

We plot in Figure~\ref{fig:phi} the optimal prices and the value function when the parameter $\phi$ varies. $\phi$ stands for the cost related to the data processing capability of platform. A higher value of $\phi$ signifies that the platform incurs greater marginal costs in expanding its data collection efforts, potentially due to constraints in technological capability. Conversely, a lower $\phi$ reflects stronger technical capacity, enabling the platform to collect and process larger volumes of data at reduced marginal cost. Since $q^*(x)$ depends only on current product quality and not on future costs, it remains unaffected.  An increase in $\phi$ reduces the value function $V$ across all states, as the platform's profit is eroded by higher storage costs.  It then leads to a lower optimal provider price. Fewer data will be collected by the platform with a greater marginal investment cost. The induced cost surpasses the revenue generated from increased data volume, prompting platform to opt for reduce data acquisition to maximize its profit.In summary, it negatively affects both the platform's profitability and its willingness to accumulation for data, while leaving consumer pricing unchanged. 

\begin{figure}[H]
    \centering
    \begin{subfigure}[b]{0.3\textwidth}
        \centering
        \includegraphics[width=\textwidth]{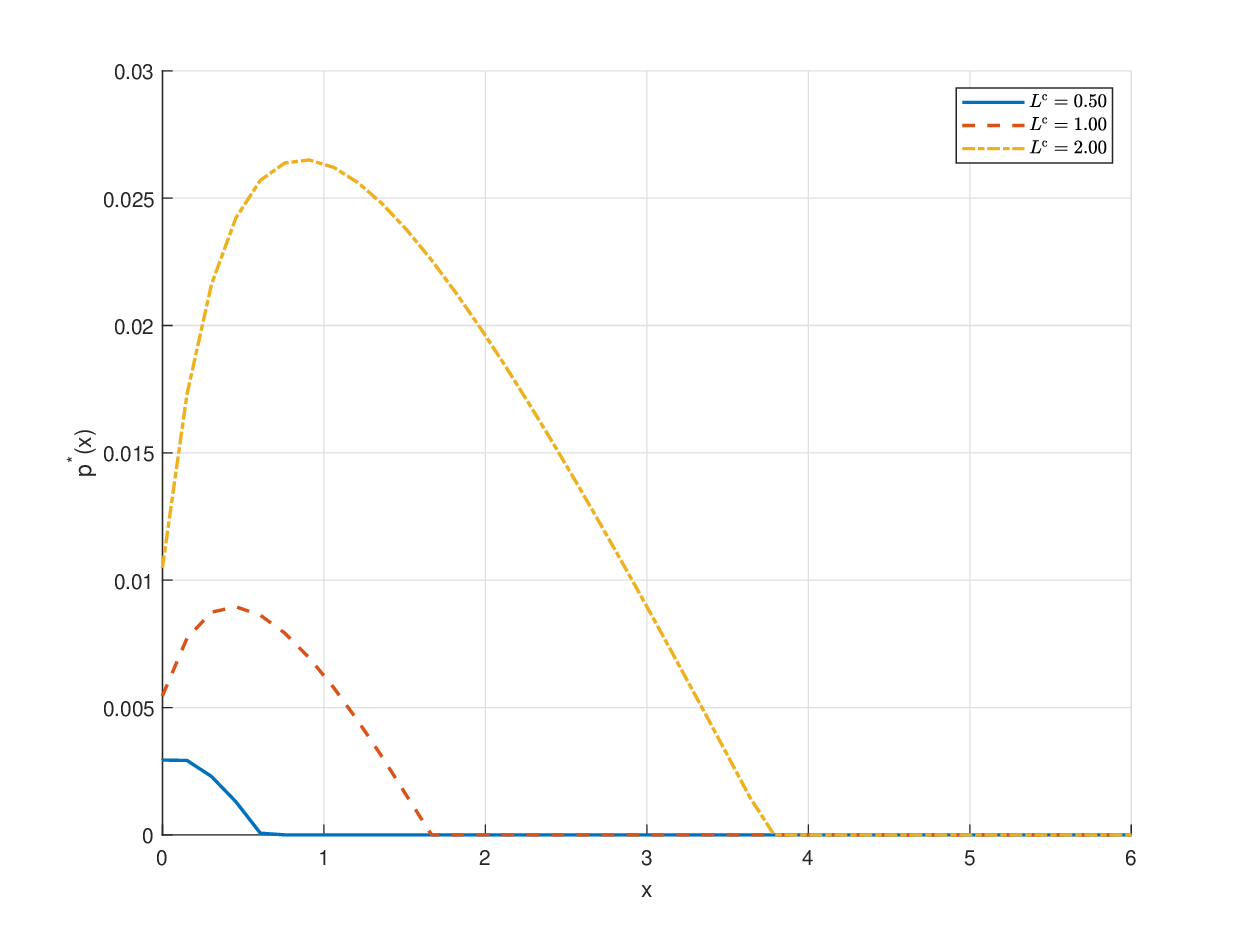}
    \end{subfigure}
    \begin{subfigure}[b]{0.3\textwidth}
        \centering
        \includegraphics[width=\textwidth]{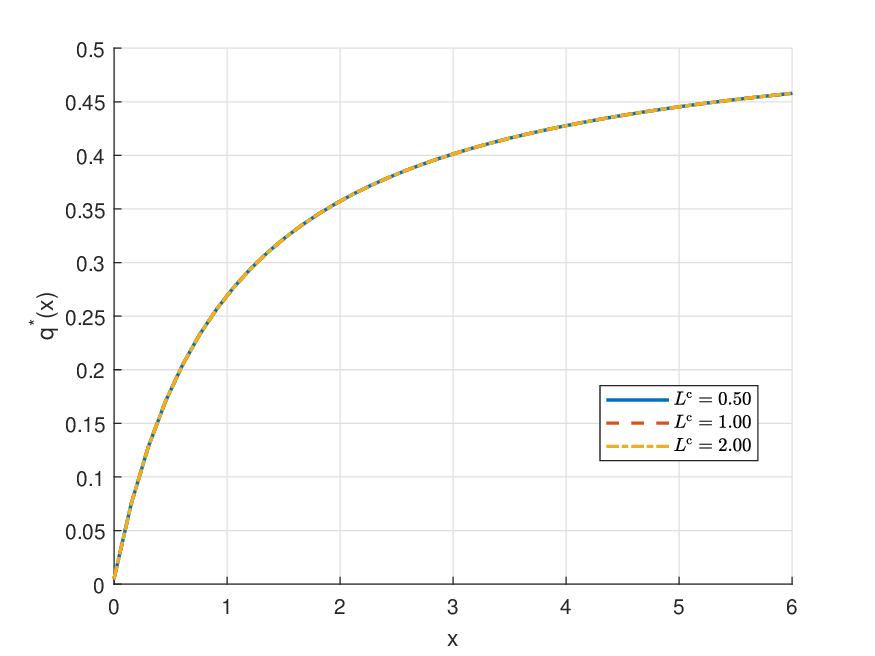}
    \end{subfigure}
    \begin{subfigure}[b]{0.3\textwidth}
        \centering
        \includegraphics[width=\textwidth]{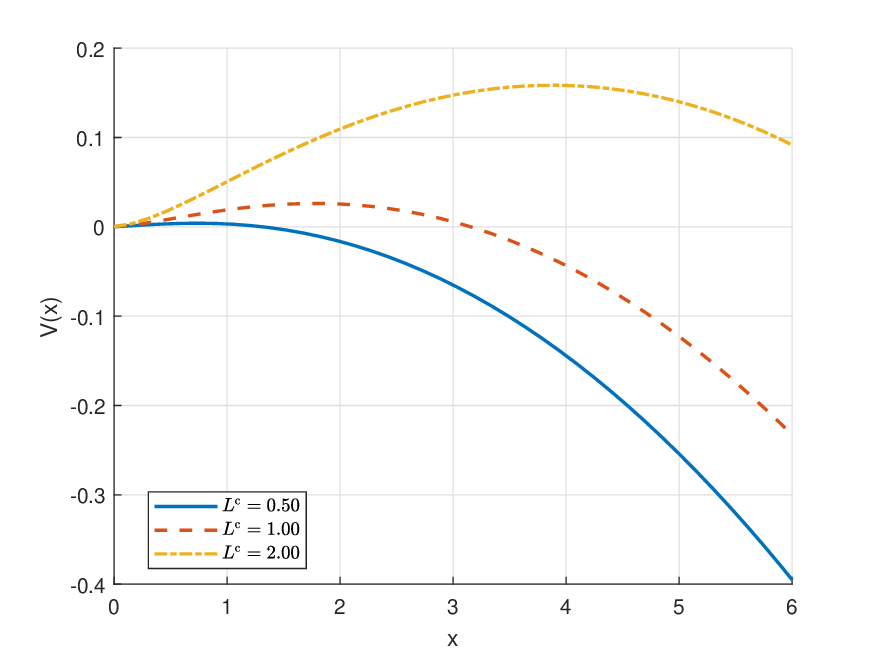}
    \end{subfigure}
    \caption{ Left panel:  optimal provider price $x\to p^*(x)$. Middle panel: optimal consumer price $x\to q^*(x)$. Right panel:  value function $ x\to V(x)$. The parameter values are $L^{\rm c}=0.5,1.0,2.0$}
    \label{fig:lambda2}
\end{figure}
In Figure~\ref{fig:lambda2}, we examine the impact of consumers' arrival excitement from the database on the pricing policies. When the platform neglects the excitement of the arrival process and considers Poisson process with constant intensity for customer arrivals, the optimal provider price $p^*(x)$ and $V(x)$ both exhibit a trend of gradual decrease.
That's because the demand from the consumer side is limited, and the marginal value of additional data diminishes as data volume increases. 
In contrast, with linear excitement, an increase in the database volume results in a proportional increase in consumer arrivals. The positive feedback between data accumulation and consumer demand creates an initial phase where data becomes more valuable, causing $V(x)$ and $p^*(x)$ to rise. Eventually, saturation of consumer arrivals and convex storage costs dominate, and $p^*(x)$ falls. This contrast highlights the role of demand‑side network effects in shaping platform data acquisition strategies.

Based on the above results, we suggest that the platform must regularly assess its own condition and the market situation to distinguish different stages. In a growth phase, engaging more customers by expanding the database scale and setting reasonable prices. Conversely, in a maturity phase, when the existing database is large enough, the platform can charge higher prices from consumers while negotiating a reduction in the provider price. By adopting these strategies, platform managers can skillfully harness the dynamic relationship between pricing and database volume, thereby achieving an optimal balance between short-term gains and long-term profitability. Moreover, the platform’s strategic priority should shift from merely acquiring additional data to leveraging data as a direct driver of consumer engagement.

\vspace{0.5cm}

\noindent{\bf Acknowledgement.} This work is supported by National Natural Science Foundation of China (No. 12471451), Natural Science Basic Research Program of Shaanxi (No. 2023-JC-JQ-05), Shaanxi Fundamental Science Research Project for Mathematics and Physics (No. 23JSZ010) and Fundamental Research Funds for the Central Universities (No. 20199235177).

\end{document}